\documentclass[12pt,reqno,a4paper]{amsart}

\usepackage{amsmath,amsthm,amssymb}
\usepackage{color}
\usepackage{enumitem}
\usepackage[mathscr]{euscript}
\usepackage[T1]{fontenc}
\usepackage{graphicx}
\usepackage{fullpage}
\usepackage[colorlinks,linkcolor=blue,citecolor=red]{hyperref}
\usepackage{mathtools}
\usepackage{multicol}
\usepackage{comment}
\usepackage{tikz}
\usetikzlibrary{cd}
\usetikzlibrary{decorations.markings}
\usepackage{wrapfig}
\usepackage[colorinlistoftodos,prependcaption]{todonotes}
\usepackage{dirtytalk}

\makeatletter
\DeclareRobustCommand{\em}{%
  \@nomath\em \if b\expandafter\@car\f@series\@nil
  \normalfont \else \bfseries\itshape \fi}
\makeatother

\newtheorem{theorem}[subsection]{Theorem}
\newtheorem{corollary}[subsection]{Corollary}
\newtheorem{lemma}[subsection]{Lemma}
\newtheorem{proposition}[subsection]{Proposition}

\theoremstyle{definition}
\newtheorem{definition}[subsection]{Definition}

\newtheorem{example}[subsection]{Example}
\theoremstyle{remark}
\newtheorem{remark}[subsection]{Remark}

\newtheorem{question}{Question}
\numberwithin{equation}{section}
\numberwithin{figure}{section}

\newcommand{\B}[1]{{\boldsymbol #1}}

\newcommand{\N}{\mathbb{N}}
\newcommand{\Z}{\mathbb{Z}}
\newcommand{\R}{\mathbb{R}}
\newcommand{\C}{\mathbb{C}}

\newcommand{\cA}{\mathcal{A}}

\newcommand{\cC}{\mathcal{C}}

\newcommand{\cH}{\mathcal{H}}
\newcommand{\cJ}{\mathcal{J}}
\newcommand{\cK}{\mathcal{K}}
\newcommand{\cP}{\mathcal{P}}
\newcommand{\cR}{\mathcal{R}}
\newcommand{\cS}{\mathcal{S}}

\newcommand{\F}[1]{{\mathfrak #1}}
\newcommand{\cp}{\C\mathrm{P}}

\newcommand{\om}{\omega}
\newcommand{\tom}{\tilde{\omega}}

\newcommand{\tJ}{\tilde{J}}

\DeclareMathOperator{\Symp}{Symp}
\DeclareMathOperator{\Diff}{Diff}
\DeclareMathOperator{\Emb}{Emb}
\DeclareMathOperator{\Stab}{Stab}

\DeclareMathOperator{\pr}{pr}

\DeclareMathOperator{\Conf}{Conf}

\DeclareMathOperator{\Aut}{Aut}

\DeclareMathOperator{\Sp}{Sp}

\DeclareMathOperator{\PU}{PU}

\DeclareMathOperator{\rel}{rel}
\DeclareMathOperator{\hofib}{hofib}

\DeclareMathOperator{\PB}{PB}

\newcommand{\IEmb}{\F I\mathrm{Emb}}
\newcommand{\into}{\hookrightarrow}

\newcommand{\del}{\partial}

\newcommand{\eoesymbol}{$\Diamond$}

\DeclareRobustCommand{\eoe}{%
  \ifmmode \mathqed
  \else
    \leavevmode\unskip\penalty9999 \hbox{}\nobreak\hfill
    \quad\hbox{\eoesymbol}%
  \fi
}

\begin{document}

\title{Embedding more than 8 symplectic balls in $\cp^2$}

\author[S. Anjos]{S\'ilvia Anjos}
\address{SA: Center for Mathematical Analysis, Geometry and Dynamical Systems \\ Department of Mathematics \\  Instituto Superior T\'ecnico \\  University of Lisbon \\ Av. Rovisco Pais \\ 1049-001 Lisboa \\ Portugal}
\email{sanjos@math.tecnico.ulisboa.pt}

\author[J. K\k{e}dra]{Jarek K\k{e}dra}
\address{JK: Department of Mathematics \\  University of Aberdeen \\ Fraser Noble Building
Aberdeen AB24 3UE \\
Scotland }
\email{kedra@abdn.ac.uk}

\author[M. Pinsonnault]{Martin Pinsonnault}
\address{MP: Department of Mathematics \\  University of Western Ontario \\ Canada}
\email{mpinson@uwo.ca}

\dedicatory{\textit{\small In fond memory of Dietmar Salamon}}


\begin{abstract}
We prove that the space of symplectic embeddings of $n\geq 1$ standard balls into the standard complex projective plane $\cp^2$, normalized so that a line has symplectic area $1$, is homotopy equivalent to the configuration space of $n$ points in $\cp^2$, provided that the sum of the ball capacities is strictly less than $1$. Our techniques further suggest that, for $n=9$, there are infinitely many homotopy types of spaces of symplectic ball embeddings, depending on the ball capacities. Moreover, for each $n\geq 5$, we exhibit capacities for which the embedding spaces are not simply connected, in contrast with the case $n \leq 4$.  As an application, we show that, for $n\geq 9$ equal balls of capacity $c<1/n$, the symplectomorphism group of the blow-up has the homotopy type of the stabilizer of $n$ distinct points in $\cp^2$.
\end{abstract}

\maketitle

\section{Introduction}
Let $(M^4, \om)$ be a closed symplectic 4-manifold.  Let $B^4(\pi r^2) =\{x\in \R^4 \ |\ \|x\|\leq r\}$ be the Euclidean ball of capacity $\pi r^2$ equipped with the standard symplectic form $\omega_0$. 
Let $\B\delta = (\delta_1,\ldots,\delta_n)$ and let
$B^4(\boldsymbol{\delta}) = B^4(\delta_1)\sqcup \cdots \sqcup B^4(\delta_n)$
be the disjoint union of $n$ symplectic balls. Let
$\Emb({\B{\delta}},M)$ denote the space of symplectic embeddings
\[
\varphi\colon B^4({\B{\delta}}) \to M
\]
equipped with the $C^{\infty}$-topology. Let $\IEmb(\B{\delta},M)$ be the space of subsets of $M$ that are images of maps belonging to $\Emb(\B{\delta},M)$, which we topologize as the quotient 
\[
\IEmb(\B{\delta},M) := \Emb(\B{\delta},M)\Big/\prod_i \Symp(B^4(\delta_i)),
\]
where the group $\Symp(B^4(\delta_i))$ acts by reparametrizations of the ball $B^4(\delta_i)$. We say $\IEmb(\B{\delta},M)$ is the space of unparametrized symplectic balls of capacities $\delta_1, \ldots, \delta_n$ in~$M$. 

Besides a handful of examples, very little is known about the homotopy type of these embedding spaces. For instance, the following question is widely open, even in the case of a single symplectic ball.

\begin{question}\label{Question}
Is $\IEmb(\B\delta,M)$ homotopy equivalent to the configuration space of $n$ points in $M$, provided the balls are small enough?  
\end{question}

Superficially, the answer seems obvious: small balls should behave like points. This intuition is correct in the more flexible case of volume-preserving embeddings (in any manifold) and, in the more rigid case of configurations of Euclidean balls in bounded regions of a Euclidean space \cite{BBK2014}. This suggests that the answer should be positive in the symplectic category as well. However, the image of an arbitrarily small symplectic ball can be arbitrarily scattered in the ambient manifold, due to the multiple transitivity of symplectic diffeomorphisms, and there is no canonical flow that takes a deformed symplectic ball to a more standard one. This makes the symplectic embedding spaces very hard to investigate. Our first result, however, is the positive answer to the above question in the case of the  complex projective plane $\cp^2$ equipped with the Fubini-Study symplectic form. 

\begin{theorem}\label{T:main}
Let $M=\cp^2$ be the complex projective plane equipped with the Fubini-Study form normalized so that the symplectic area of a line is $1$. Let $n\in \N$ be a positive integer. If $\delta_1 + \cdots + \delta_n < 1$  then
\[
\IEmb(\B\delta,\cp^2) \simeq {\rm Conf}(n,\cp^2).
\]
\end{theorem}

An immediate corollary to Theorem~\ref{T:main} is that, for any $n\geq 3$, the symnplectomorphism group of the $n$-fold blow up of $\cp^2$ at equal balls of capacity $c<1/n$ is homotopy equivalent to the stabilizer of $n$ distinct points in $\cp^2$, see~Proposition~\ref{prop:homotopy type Symp}.\\

A closely related question about the topology of embedding spaces is the following.

\begin{question} \label{Question2}
Given a symplectic manifold $(M,\omega)$, is it possible to find an infinite sequence of capacities $\B\delta_k$ such that the homotopy types of the embedding spaces $\Emb(\B{\delta_k},M)$ are pairwise distinct? 
\end{question}

In the case of $\cp^2$ it is proved in \cite{Anjos-al-Stability} that, for $n < 9$,  $\Emb(\B{\delta},\cp^2)$  exhibits only finitely many distinct homotopy types. The second main result of this paper is a step towards a positive answer to the above question. 

The ordered set ${\B\delta} =(\delta_1,\ldots,\delta_n)\in \R^n$ of capacities is called {\it admissible} if there is a symplectic embedding $B^4(\B\delta)\to \cp^2$. The set of all such $n$-tuples forms an open convex subset $\cC_n\subseteq \R^n$, which is subdivided into convex regions called {\it chambers}. Given sets of admissible capacities $\B\delta = (\delta_1,\ldots,\delta_n)$ and $\B{\delta'}=(\delta'_1,\ldots,\delta'_n)$, we say that $\B\delta\leq \B{\delta'}$ if $\delta_i\leq \delta'_i$ for all $i=1,\ldots, n$.

\begin{theorem}\label{T:sequence}
Let $n=9$. There exist a sequence $\{C_k\}_{k\in \N}$ of pairwise adjacent chambers in $\cC_n$ and a corresponding sequence $\{\B\delta_k\}_{k\in \N}$ of capacities $\B\delta_k\in C_k$  such that 
\begin{enumerate}
\item $C_0 =\Delta_0=\{(\delta_1,\ldots,\delta_n)\in \R^n\  |\ \delta_1+\cdots +\delta_n<1\}$;\item $\lim_{k\to \infty} \B\delta_k = (1/3,1/3,\ldots,1/3)$; and
\item for any pair of consecutive chambers $C_k$ and $C_{k+1}$, we can find capacities $\B\delta \in C_k$ and $\B{\delta'} \in C_{k+1}$ satisfying $\B\delta < \B{\delta'}$, and for any such capacities, the restriction map 
\[
\Emb(\B{\delta'},\cp^2)\to \Emb(\B{\delta},\cp^2)
\]
is not a homotopy equivalence.
\end{enumerate}
\end{theorem}

The accumulation of chambers near the monotone point, as stated in item (2) above, is a special case of a more general phenomenon that holds for any rational $4$-manifold: all accumulation points  lie on the \emph{round boundary} of the symplectic cone, namely on the light cone $\sum_i\delta_i^2=1$. See Proposition~\ref{prop:accumulation points} for a precise statement.

Notice that item~$(3)$ of the above theorem does not imply that the spaces $\Emb(\B{\delta'},\cp^2)$ and $\Emb(\B\delta,\cp^2)$ are not homotopy equivalent. Nevertheless, we conjecture that this is indeed the case. As evidence, we combine our results with the recent work of Li, Li, and Wu~\cite{LLW-Torelli} to construct a path through finitely many stability chambers, starting in a chamber where the embedding space is simply connected and ending in one where it is not; see Examples~\ref{ex:non simply connected 1} and~\ref{ex:non simply connected 2}.

\subsection*{Outline of the proof of Theorem~\ref{T:main}.} 
Let $\Symp(B^4(\B{\delta}),0)$ be the group of symplectic diffeomorphisms
of the disjoint union of balls fixing the center of the balls. An element of the quotient 
\[
\IEmb^*(\B\delta,\cp^2) = \Emb(\B\delta,\cp^2)/\Symp(B^4(\B{\delta},0))
\]
is an ordered configuration of symplectic balls {\it with centers}. Consequently, there is a natural map
\[
\IEmb^*(\B\delta,\cp^2) \to {\rm Conf}(n,\cp^2)
\]
given by $[\varphi] \mapsto (\varphi_1(0),\ldots,\varphi_n(0))$,
where $\varphi_i\colon B(\delta_i)\to \cp^2$ is the symplectic embedding
of the $i$-th ball. On the other hand, the natural projection $\IEmb^*(\B\delta,\cp^2)\to \F \IEmb(\B{\delta},\cp^2)$ is a homotopy equivalence, according to \cite[Section 2.2]{AKP2024}. Thus the homotopy equivalence in the statement of Theorem \ref{T:main} is the composition
\[
\F \IEmb(\B{\delta},\cp^2)\to \IEmb^*(\B\delta,\cp^2)\to {\rm Conf}(n,\cp^2),
\]
where the first map is the homotopy inverse of the natural projection.

By the symplectic blow-up construction, every set of admissible capacities $\B\delta$ defines a cohomology class $[\omega_{\B{\delta}}]$ in the symplectic cone of the $n$-fold complex blow-up $\widetilde{M}_n = \cp^2\#n\overline{\cp}^2$ (note that the actual blow-up form $\omega_{\B{\delta}}$ is only defined up to isotopy). The Poincar\'e dual of $[\omega_{\B{\delta}}]$ is given by
\[
{\rm PD}[\omega_{\B{\delta}}] = L - \sum_{i=1}^n \delta_i E_i \in H_2(\widetilde{M}_n;\R),
\]
where $L$ is the class of a line and $E_i$ are the classes of exceptional divisors. Thus the set 
\[
\cC_n =\left\{\B{\delta}=(\delta_1,\ldots,\delta_n) \in \R^n ~|~  \B\delta \ \mbox{is admissible} \right\}
\] 
of admissible capacities can also be considered as a subset of $H_2(\widetilde{M}_n;\R)$ via the inclusion $\B\delta \mapsto {\rm PD}[\omega_{\B{\delta}}]$. Notice that this inclusion is a restriction of an affine map $\R^n\to H_2(\widetilde{M}_n;\R)$.

Let $\Delta_0=\left\{(\delta_1,\ldots, \delta_n)\in \R^n\ |\ \delta_1+\cdots+\delta_n < 1\right\}$. The main ingredient in the proof of Theorem \ref{T:main} is to show that $\Delta_0$ is contained in a stability component, where by a {\it stability component} we mean a maximal connected subset of $\cC_n$ on which the homotopy type of $\IEmb(\B\delta,\cp^2)$ is constant. This step uses the framework of~\cite[Section 2]{LP2004}, \cite[Sections 2.3 and 2.4]{Anjos-al-Stability}, and~\cite[Section 5.3]{AKP2024}, which reduces the question to showing that certain spaces of almost complex structures, depending on the capacities $\B\delta\in\cC_n$, are in fact constant over $\Delta_0$.

This reformulation is based on two natural fibrations with homotopy equivalent fibers. The first one is induced by the natural action of symplectomorphisms on embeddings, namely,
\begin{equation}
\begin{tikzcd}
\Stab(\phi(B^4(\B\delta))) \ar[r] &\Symp(\cp^2)\ar[r] &\IEmb(\B\delta,\cp^2),
\end{tikzcd}
\label{Eq:diagram-Embs-part1}
\end{equation}
where the left-hand side entry is the isotropy subgroup of the ball  
$\phi(B^4(\B\delta))\subset \cp^2$. To describe the second one, let $\Diff_h(\widetilde{M}_n,\Sigma)$ denote the group of diffeomorphisms of the complex blow-up $\widetilde{M}_n$ which act trivially on homology and which leave the exceptional divisors invariant. It acts on a set $\cA(\B\delta,\Sigma)$ of almost complex structures on $\widetilde{M}_n$ for which the exceptional divisors are embedded $J$-holomorphic spheres. The homotopy fiber of the evaluation map is equivalent to the subgroup $\Symp_h(\B\delta, \Sigma)$ consisting of diffeomorphisms of $\widetilde{M}_n$ preserving $\omega_{\B\delta}$ and sending $\Sigma$ to itself. This yields the second (homotopy) fibration:
\begin{equation}
\begin{tikzcd}
\Symp_h(\B\delta, \Sigma)\ar[r] & \Diff_h(\widetilde{M}_n,\Sigma)\ar[r]
& \cA(\B\delta,\Sigma).
\end{tikzcd}
\label{Eq:diagram-ACS-part1}
\end{equation}
By~\cite[Section 2]{LP2004}, the stabilizer group $\Stab(\phi(B^4(\B\delta)))$ is homotopy equivalent to the homotopy fiber $\Symp_h(\B\delta,\Sigma)$, and by~\cite[Theorem 1.3]{Anjos-al-Stability} the homotopy type of the embedding space $\IEmb(\B\delta,\cp^2)$ changes precisely when the homotopy type of the space $\cA(\B\delta, \Sigma)$ changes. 

The final step of the proof is to show that all spaces $\F \IEmb(\B{\delta},\cp^2)$ for $\B\delta\in \Delta_0$ are homotopy equivalent to ${\rm Conf}(n,\cp^2)$. This conclusion essentially follows from~\cite[Lemma 2.2]{Anjos-al-Stability}.

\subsection*{Outline of the proof of Theorem \ref{T:sequence}}
In the first part of the proof we show in Section~\ref{S:1/3} that there is an infinite family of hyperplanes, called walls, associated with certain homology classes of negative self-intersection, intersecting the set $\cC_9$ of admissible capacities, and defining a partition of $\cC_9$ into convex chambers in which stability holds. Then we show that there exist sequences of chambers $\{C_k\}_{k\in \N}$ and of capacities $\B\delta_k\in C_k$, starting with $C_0=\Delta_0$, and such that $\B\delta_k$ and $\B\delta_{k+1}$ are separated by exactly one wall associated with a $(-2)$-curve. Moreover, the sequence $\{\B\delta_k\}$ converges to the point $(1/3,1/3,\ldots,1/3)$.

Let $\B\delta'>\B\delta$ and let 
\[
\Emb(\B\delta',\cp^2)\to \Emb(\B\delta,\cp^2)
\]
be the map defined by restriction of an embedding of $B^4(\B\delta')$ to smaller balls $B^4(\B\delta)$. By taking the mapping cylinder, we can consider the pair $(\Emb(\B\delta,\cp^2),\Emb(\B\delta',\cp^2))$. To show that the restriction map is not a homotopy equivalence, we prove that a homotopy group of the pair is nontrivial. We adapt the framework used in the proof of Theorem~\ref{T:main} to make it better suited to the study of restriction maps. The main idea is to track more closely the almost complex structures and the diffeomorphisms in the blow-up construction. In this setting, the fibration~\eqref{Eq:diagram-Embs-part1} is replaced by
\begin{equation}
\begin{tikzcd}
{\rm Stab}(\phi(\B\delta)) \ar[r] &\Symp(\cp^2)\ar[r] &\Emb(\B\delta,\cp^2),
\end{tikzcd}
\label{Eq:diagram-Embs}
\end{equation}
where the space of images is replaced by actual embeddings, and where the fiber is the isotropy subgroup of a given embedding $\phi\colon B^4(\B\delta)\to \cp^2$. Let $\Diff_h(\widetilde{M}_n,\rel \Sigma)$ denote the group of diffeomorphisms of the complex blow-up $\widetilde{M}_n$, which act trivially on homology and are supported away from the exceptional divisor. It acts on the set $\cA(\B\delta,\rel \Sigma)$ of almost complex structures on $\widetilde{M}_n$ which are standard in a neighborhood of the exceptional divisor; see Section~\ref{S:preliminaries} for a precise definition. As before, these spaces fit into a homotopy fibration analogous to~\eqref{Eq:diagram-ACS-part1}.
\begin{equation}
\begin{tikzcd}
\Symp_h(\B\delta, \rel \Sigma)\ar[r] & \Diff_h(\widetilde{M}_n,\rel \Sigma)\ar[r]
& \cA(\B\delta,\rel \Sigma).
\end{tikzcd}
\label{Eq:diagram-ACS}
\end{equation}
This allows one to prove that the non-triviality of  $\pi_k(\Emb(\B\delta,\cp^2),\Emb(\B\delta',\cp^2))$
is equivalent to the non-triviality of
$\pi_k(\cA(\B\delta,\rel \Sigma),\cA(\B\delta',\rel \Sigma))$. This part of the proof is not immediate: as the symplectic blow-up is not functorial, it requires one to verify compatibility conditions in a composition of a number of other homotopy equivalences. The use of the space $\cA(\B\delta,\rel \Sigma)$ is crucial for that part of the proof. The details are presented in the Appendix. 

In order to conclude the proof of Theorem~\ref{T:sequence}, let $\B {\delta'} \in C_{k+1}$ and $ \B\delta \in C_k$ be two capacities satisfying $\B {\delta'} > \B {\delta}$  which do not belong to any wall and such that there is exactly one wall associated with a $(-2)$-curve between them. Then
\[
\cA({\B\delta},\rel \Sigma) = \cA({\B\delta'},\rel \Sigma) \sqcup \cA_2,
\]
where $\cA_2 \subseteq \cA(\B\delta,\rel \Sigma)$ is a submanifold of codimension $2$.  We show in Lemma \ref{L:eells} that this implies that the relative group $\pi_2(\cA(\B\delta,\rel \Sigma),\cA(\B\delta',\rel \Sigma))$ is non-trivial, which, according to the discussion above, shows that
$\pi_2\big(\Emb(\B\delta,\cp^2),\Emb(\B\delta',\cp^2)\big)$ is non-trivial. It follows that the restriction map $\Emb(\B\delta',\cp^2)\to \Emb(\B\delta,\cp^2)$ cannot be a homotopy equivalence. 
\subsection*{Organization of the paper}
Section~\ref{S:preliminaries} reviews the background on stability of embedding spaces in rational 4-manifolds. Sections~\ref{S:main} and \ref{S:wall-crossing} are devoted to the proof of the two main theorems: Section~\ref{S:main} establishes Theorem~\ref{T:main}, while Section 4 develops the wall-crossing machinery behind Theorem~\ref{T:sequence}.
Section \ref{S:1/3} classifies the homology classes defining walls for $n = 9$ and  reformulates wall-crossing as Mozes' game of numbers in order to complete the proof of Theorem \ref{T:sequence}. It ends with a general result on the accumulation of walls in any rational 4-manifold. Section \ref{S:applications} derives applications to symplectomorphism groups of blow-ups of $\mathbb{CP}^2$, and Section \ref{S:example} combines the results with recent work of Li, Li and Wu on on the symplectic Torelli group of symplectic rational surfaces. The Appendix supplies the technical details relating symplectic blow-ups and embeddings that justify the compatibility statements used in Section~\ref{S:wall-crossing}.

\subsection*{Acknowledgements.}
The authors would like to thank Pedro Boavida de Brito for helpful conversations. MP is supported by NSERC Discovery Grant RGPIN-2020-06428. The three authors are funded by Fundação para a Ciência e Tecnologia (FCT), Portugal, through grant No. UID/04459/2025 and project 2023.13969.PEX, and by the Recovery and Resilience Plan (PRR) through grant No. UID/PRR/04459/2025.

\section{Background results  on stability of embedding spaces}\label{S:preliminaries} 
We recall the basic approach used in the study of embedding spaces on symplectic rational $4$-manifolds and, in particular, the central notion of stability. For a more complete exposition, the reader can consult~\cite[Sections 2.3 and 2.4]{AKP2024}. Although the proofs of Theorems~\ref{T:main} and~\ref{T:sequence} use two closely related frameworks -- namely, the spaces $\cA(\B\delta, \Sigma)$ for the study of images $\IEmb(\B\delta,\cp^2)$ and their relative counterparts $\cA(\B\delta, \rel \Sigma)$ for embeddings $\Emb(\B\delta,\cp^2)$ -- we state the results only for the latter in order to streamline the exposition. Similar results hold for the former case.

\begin{definition}
Let $(M,\om)$ be a $4$-manifold and let $\B\delta$, $\B{\delta'}$ be two sets of capacities in $\cC_n$. We say that {\it $\B\delta$ and $\B{\delta'}$ are in the same  stability component} if there exists a continuous family of capacities $\B\delta_t\subset\cC_n$ interpolating $\B\delta$ and $\B{\delta'}$ for which the homotopy type of $\Emb(\B\delta_t,M)$ is constant. 
\end{definition}

For a rational $4$-manifold $(M,\om)$, given capacities $\B\delta= (\delta_1, \ldots, \delta_n) \in \cC_n$, the connectedness of $\Emb(\B\delta,M)$~\cite[Corollary 1.5]{McDuff-DeformationsToIsotopy} implies that the natural action on $\Emb(\delta,M)$ of the group $\Symp_h(M, \omega)$ of symplectomorphisms acting trivially on homology is transitive.  
Similarly to the fibration \eqref{Eq:diagram-Embs}, evaluating this action at an embedding $\phi$ defines a fibration
\begin{equation}\label{second fibration parametrized}
 \Stab_h(\phi) \to \Symp_h(M,\omega) \to \Emb(\B\delta,M),
\end{equation}
where the fiber is the isotropy subgroup of an embedding $\phi: B^4(\B\delta) \to M$.  

As explained in the outline of the proof of Theorem \ref{T:sequence}, this fibration, together with the homotopy fibration \eqref{Eq:diagram-ACS}, can be used to investigate the homotopy type of $\Emb(\B\delta, M)$.  More precisely, as the capacities $\B\delta$ vary, the homotopy types of $\Symp_h(\B\delta, \rel \Sigma)$ and $\Emb(\B\delta,M)$ change precisely when the homotopy type of the space $\cA(\B\delta,\rel \Sigma)$ changes.  
Here, $\cA(\B\delta,\rel \Sigma)$ denotes the space of almost complex structures on the blow-up $\widetilde{M}_n$ that are tamed by some symplectic form in $\Omega(\B\delta,\rel \Sigma)$, and that coincide with the standard structure $\widetilde{J}_0$ on a neighborhood of $\Sigma$. The space $\Omega(\B\delta,\rel \Sigma)$ consists of symplectic forms on $\widetilde{M}_n$ in the class $[\omega_{\B\delta}]$ which agree with $\omega_{\B\delta}$ on a neighborhood of $\Sigma$. For more details about what it means to be standard near $\Sigma$, see the Appendix.

As mentioned in the introduction, the use of the space $\cA(\B\delta,\rel \Sigma)$ is crucial. The main idea is that, using tame-to-tame inflation (see~\cite{McDuff-J-inflation, Buse-Neg-inflation, CPP-JTtameInflation}), one can show that, under suitable conditions on the capacities, we have
\[
\cA(\B\delta,\rel \Sigma)= \cA(\B\delta',\rel \Sigma),
\]
which in turn implies that the two sets of capacities $\B\delta$ and $\B\delta'$ belong to the same stability component.

For rational $4$-manifolds with Euler number $\chi \leq 12$, the process of finding suitable curves for $J$-inflation was systematically studied by W. Zhang~\cite{Zhang-CurveCone} in his proof of the almost K{\"a}hler Nakai--Moishezon duality theorem (also based on the inflation lemma). In order to state the relevant results, recall that, given an almost complex structure $J$ on $M$, the curve cone $A_J(M)\subset H_2(M,\R)$ is defined as
\[
A_J(M) =\left\{\sum_i a_i[C_i]~|~a_i>0  \right\}
\]
where the $C_i$ are irreducible $J$-holomorphic curves. We write $\overline{A}_J(M)$ for its closure, and $\overline{A}^{\vee,>0}_J(M)$ for the positive dual of the closure. We also define the tame and compatible K{\"a}hler cones of $J$ as
\begin{align*}
\cK_J^t &= \left\{ \alpha\in H^2(M,\R)~|~J \text{~is tamed by some symplectic form in class~} \alpha \right\}\\
\cK_J^c &= \left\{ \alpha\in H^2(M,\R)~|~J \text{~is compatible with some symplectic form in class~} \alpha \right\}.
\end{align*}
We can now state a version of Zhang's Nakai--Moishezon duality theorem that is suitable for our purposes. 
\begin{theorem}[\cite{Zhang-CurveCone}, Theorems~1.1 and~1.6]\label{thm:Nakai-Moishezon duality}
If $J$ is tamed by some symplectic form on $M = \cp^2\#k\overline{\cp^2}$ with $2\leq k\leq 9$, then the boundary of the curve cone is generated by curves of self-intersection $\leq -1$. Moreover, 
we have equalities
\[\cK_J^t = \cK_J^c = \overline{A}^{\vee,>0}_J(M)\]
and any two classes in the tame K{\"a}hler cone $\cK_J^t$ can be joined by a path obtained by $J$-tame inflation along $J$-holomorphic curves.
\end{theorem}
Assuming this result, we can describe the stability of embedding spaces as follows.  
Given \(J \in \cA(\B\delta,\rel \Sigma)\), let \(\cS(J)^{\leq -1}\) denote the set of classes represented by embedded \(J\)-holomorphic spheres with negative self-intersection. Then, by performing \(J\)-tame inflation, one can reach any other symplectic class for which all classes in \(\cS(J)^{\leq -1}\) have strictly positive symplectic area. In other words, \(J\)-tame inflation is constrained only by the positivity of the symplectic pairing on the classes in \(\cS(J)^{\leq -1}\). In particular, this yields the following convenient criterion for the inclusion
\[
\cA(\B\delta,\rel \Sigma) \subset \cA(\B{\delta'},\rel \Sigma).
\]

\begin{corollary}\label{cor:criterion for inclusion}
Let $J\in\cA(\B\delta,\rel \Sigma)$. If $\B{\delta'}$ is another set of admissible capacities for which the class $[{\om}_{\B{\delta'}}]$ evaluates strictly positively on all classes in $\cS(J)^{\leq -1}$, then $J\in \cA(\B{\delta'},\rel \Sigma)$. 
\end{corollary}

This criterion allows us to avoid explicitly describing the inflation procedures required to establish stability of the embedding spaces. In particular, it provides a practical way to determine whether a point on the boundary of an open stability component belongs to that component.

To make this criterion more effective, we define the notion of a wall for a rational symplectic $4$-manifold $(M,\om)$ as follows. First, define 
\begin{align*}
\cS(\B\delta)^{\leq-1} &:= \bigcup_{J\in\cA(\B\delta,\rel \Sigma)} \cS(J)^{\leq -1}\\
\cS(M,\om)^{\leq -1} &:= \bigcup_{\B\delta\in\cC_n} \cS(\B\delta)^{\leq -1}
\end{align*}
Each $A\in \cS(M,\om)^{\leq -1}\subset H_2(\widetilde{M}_n;\Z)$ defines a linear functional
\[
\ell_A\colon H_2(\widetilde{M}_n;\R)\to \R, \quad \ell_A(B) = A\cdot B.
\]
The kernel of $\ell_A$ is called the \textit{wall} corresponding to $A$. It is a hyperplane in $H_2(\widetilde{M}_n;\R)$, denoted by $H_A$. These hyperplanes may or may not intersect the region $\cC_n$ of ordered admissible capacities.

\begin{corollary}\label{cor:no chambers in walls}
Let $\B\delta$ be a set of admissible capacities lying on the intersection of the walls defined by the classes $A_1,\ldots,A_k$, that is, such that $\om_{\B\delta}(A_i)=0$ for all $i$. Then $\B\delta$ belongs to the stability component in which the classes $A_i$ have strictly negative symplectic areas. In particular, each stability component for the embedding problem has nonempty interior.
\end{corollary}

\begin{remark}\label{remark:Relations to LLW}
Stability components for the homotopy type of symplectomorphism groups of rational $4$-manifolds are investigated in~\cite{Anjos-al-Stability} and~\cite{LLW-Torelli}. Although the techniques are similar to those used for embedding spaces, the two problems differ in an essential way. In the case of symplectomorphism groups, since every symplectic form is Cremona equivalent to a reduced one, one can restrict the discussion to symplectic forms representing \emph{reduced} classes. As Cremona transformations correspond to reflections along $(-2)$ walls, it follows that these walls bound the set of reduced classes -- they are \emph{exterior walls} -- and may themselves contain proper stability components of strictly positive codimension. However, for the embedding problem on $\cp^2$, Cremona transformations are no longer available, as the only Cremona transformation that leaves the classes $E_1,\ldots, E_n$ invariant is the identity. It follows that $(-2)$ walls are \emph{interior} walls of the set of admissible capacities and, by Corollary~\ref{cor:no chambers in walls}, never contain proper stability components. Actually, for the embedding problem on $\cp^2$, the set of admissible capacities $\cC_n$ is the $K$-symplectic cone of the $n$-fold blow-up, and the exterior walls are all given by $(-1)$ classes.\eoe
\end{remark}

\section{Proof of Theorem \ref{T:main}}\label{S:main}

Recall that, given sets of capacities $\B\delta = (\delta_1,\ldots,\delta_n)$ and $\B{\delta'}=(\delta'_1,\ldots,\delta'_n)$, we write $\B\delta\leq \B{\delta'}$ if $\delta_i\leq \delta'_i$ for all $i=1,\ldots, n$. Let $(M_{\B{\delta}}, \tom_{\B{\delta}})$ denote the symplectic blow-up of $\cp^2$ along balls of capacities $\B{\delta}=(\delta_1, \ldots, \delta_n)$. Let $\cJ(\tom_{\B\delta}, \Sigma)$ be the space of almost complex structures $J$ tamed by $\tom_{\B\delta}$ and such that $\Sigma$ is $J$-holomorphic. Finally, let $\cA(\B\delta, \Sigma)$ be the space of almost complex structures $J$ tamed by some symplectic form cohomologous to $\tom_{\B\delta}$ and for which $\Sigma$ is $J$-holomorphic. We work with $\cA(\B\delta,\Sigma)$, with no constraint
near $\Sigma$, because increasing capacities requires inflation along
curves meeting $\Sigma$, see Remark~\ref{remark:note on spaces A and inflation}.

\begin{lemma}\label{lemma:decreasing capacities}
If $\B\delta\leq \B{\delta'}$ then there is an inclusion $\cA(\B{\delta'}, \Sigma) \subset \cA(\B {\delta}, \Sigma).$
\end{lemma}
\begin{proof}
Since, by assumption, the exceptional classes $E_i$ are represented by embedded $J$-holomorphic spheres for all $J\in \cA(\B {\delta'},\Sigma)$, we can perform $J$-tamed inflation along these representatives to decrease the capacities $\delta_i'$ to any given numbers $0<\delta_i<\delta_i'$.
\end{proof}

\begin{lemma}\label{lemma:existence fiber classes}
For any $J\in\cA(\B{\delta}, \Sigma)$, the fiber class $F_i=L-E_i$ is represented by embedded $J$-holomorphic spheres. 
\end{lemma}
\begin{proof}
Since any two symplectic forms cohomologous to $\tom_{\B\delta}$ are diffeomorphic, and because any two embedded representatives of an exceptional class $E_i$ are symplectically isotopic, it suffices to prove the claim for all $J\in \cJ(\tom_{\B\delta}, \Sigma)$.

Since $k(F_i)=(F_i^2+c_1(F_i))/2 = 1$, Taubes' theorem implies that for a generic pair $(p,J)$ in $M_{\B{\delta}}\times \cJ(\tom_{\B\delta}, \Sigma)$, there is exactly one embedded $J$-holomorphic representative passing through $p$, and that the set of generic pairs is dense in $M_{\B{\delta}}\times \cJ(\tom_{\B\delta}, \Sigma)$. By the adjunction formula, such a representative is a sphere. Pick an arbitrary pair $(p,J)$ in $M_{\B{\delta}}\times \cJ(\tom_{\B\delta}, \Sigma)$ and consider a sequence of generic pairs $(p_j,J_j)\in M_{\B{\delta}}\times \cJ(\tom_{\B\delta}, \Sigma)$ converging to $(p,J)$. Let $S_j$ be the embedded fiber corresponding to $(p_j,J_j)$. By Gromov's convergence theorem, there is a subsequence converging to a $J$-holomorphic sphere $S$ or to a reducible cusp curve $C=\cup_k C_k$ passing through $p$. In the former case, the adjunction formula implies that $S$ must be embedded, which proves the statement. In the latter case, because the classes $E_i$ are all represented by $J$-holomorphic spheres, the only way the class $F_i=L-E_i$ can decompose into $J$-holomorphic representatives is as a union of the form $\big(L-E_i-\sum_\ell E_\ell \big) \cup_\ell E_\ell$. Since each component is of negative self-intersection, such a curve is isolated. Consequently, there must exist embedded $J$-holomorphic representatives of $F_i$ that contain other generic points $q\not\in C$.
\end{proof}

\begin{proposition}\label{prop:Stability holds in small chamber}
For any $\B\delta, \B{\delta'}$ in the region $\Delta_0$, we have $\cA(\B\delta,\Sigma)=\cA(\B{\delta'}, \Sigma)$.
\end{proposition}
\begin{proof}
First observe that $\Delta_0$ is a simplex whose vertices are the origin $\B 0$ and the standard basis vectors $e_i$ which can be identified with the classes $F_i$.

Let $\B\delta, \B{\delta'}\in \Delta_0$ and let $J\in \cA(\B\delta,\Sigma)$. It follows from Lemma~\ref{lemma:decreasing capacities} that the $J$-tame inflation brings $\B\delta$ to $\B{\delta_{\epsilon}}$, which is arbitrarily close to the origin, and consequently that $J\in \cA(\B{\delta_{\epsilon}},\Sigma)$. Similarly, by Lemma~\ref{lemma:existence fiber classes}, the fiber classes $F_i$ are represented by embedded $J$-holomorphic spheres and the inflation along the classes $F_i$ moves $\B{\delta_{\epsilon}}$ in the $e_i$ direction. Thus performing suitable $J$-tame inflations along the classes $F_i$ brings $\B{\delta_{\epsilon}}$ to $\B{\delta'}$ showing that $J\in \cA(\B{\delta'},\Sigma)$. Since $J\in \cA(\B\delta,\Sigma)$ was arbitrary, this shows that $\cA(\B\delta,\Sigma)\subseteq \cA(\B{\delta'},\Sigma)$. An analogous procedure from $\B{\delta'}$ to $\B{\delta}$ shows the required equality.
\end{proof}

\begin{proof}[Proof of Theorem \ref{T:main}]
By~\cite[Theorem 1.3]{Anjos-al-Stability} the homotopy type of the embedding space $\IEmb(\B\delta,\cp^2)$ changes precisely when the homotopy type of the space $\cA(\B\delta, \Sigma)$ changes. Consequently, by Proposition~\ref{prop:Stability holds in small chamber}, stability holds in the region $\Delta_0$.

We now determine the homotopy type of $\Emb(\B\delta,\cp^2)$. Let
\[
\Sp(4;\R)^n \to {\rm SpConf}(n,\cp^2)\to {\rm Conf}(n,\cp^2)
\]
be the symplectic principal bundle associated with the tangent bundle of $\cp^2$. An element of the total space is a configuration of points together
with a symplectic frame at each point. Let $\Psi_{\B{\delta}}\colon \Emb(\B{\delta},\cp^2) \to {\rm SpConf}(n,\cp^2)$ be defined by
\[
\Psi_{\B{\delta}}(\varphi) = d\varphi(F_0),
\]
where $F_0$ consists of the standard symplectic frames at the center of each ball $B(\delta_i)$, for $i=1,\ldots,n$.

The set of admissible capacities is a poset with respect to the relation $\B{\delta}\leq \B{\delta'}$ given by $\delta_i\leq \delta'_i$ for all $i=1,\ldots,n$. This implies that the spaces of embeddings $\Emb(\B{\delta},\cp^2)$ form a directed system with respect to the maps ${\rm res}_{\B{\delta'},\B{\delta}}\colon \Emb(\B{\delta'},\cp^2)\to \Emb(\B{\delta},\cp^2)$ induced by restrictions of embeddings if $\B{\delta}\leq \B{\delta'}$. We have $\Psi_{\B{\delta'}} = \Psi_{\B{\delta}}\circ {\rm res}_{\B{\delta'},\B{\delta}}$ and hence there is a well-defined map
\[
\lim_{\longrightarrow} \Emb(\B{\delta},\cp^2) \to {\rm SpConf}(n,\cp^2),
\]
which is a weak homotopy equivalence, according to \cite[Lemma 2.2]{AKP2024}.

When restricted to $\Delta_0$, stability implies that the directed system $ \Emb(\B{\delta},\cp^2)$ is constant. Hence we get that
\[
\Psi_{\B{\delta}}\colon \Emb(\B{\delta},\cp^2)\to {\rm SpConf}(n,\cp^2)
\]
is a homotopy equivalence for all $\B{\delta}\in \Delta_0$. Consider the following diagram of fibrations.
\[
\begin{tikzcd}
\Symp(B^4(\B{\delta}),0) \ar[d]\ar[r] & \Sp(4;\R)^n\ar[d]\\
\Emb(\B{\delta},\cp^2)\ar[d]\ar[r,"\Psi_{\B{\delta}}"]        & {\rm SpConf}(n,\cp^2)\ar[d]\\
\IEmb^*(\B{\delta},\cp^2) \ar[r,"{\rm center}"]         & {\rm Conf}(n,\cp^2)
\end{tikzcd}
\]
where $\Symp(B^4(\B{\delta}),0)$ is the subgroup of $\Symp(B^4(\B{\delta}))$ fixing the origins $0_i \in B^4(\delta_i)$. We have the following sequence of maps
\[
\Sp(4;\R)^n  \to {\rm U}(2)^n  \subset \Symp(B^4(\B{\delta}),0)  \subset \Symp (B^4(\B{\delta})),
\]
where the first map is a deformation retraction. Since all maps are homotopy equivalences (see \cite[Section 2.2]{AKP2024} for the last inclusion), we obtain that 
\[
{\rm center}\colon \IEmb^*(\B{\delta},\cp^2)\to {\rm Conf}(n,\cp^2)
\]
is a homotopy equivalence. Furthermore, the projection 
$\IEmb^*(\B{\delta},\cp^2)\to \F \IEmb(\B{\delta},\cp^2)$ is a homotopy equivalence, according to \cite[Section 2.2]{AKP2024}, which finishes the proof of Theorem~\ref{T:main}.
\end{proof}

\section{Wall crossing}\label{S:wall-crossing}
Recall that  $\Omega(\B\delta, \rel \Sigma)$ is  the space of symplectic forms on  $\widetilde{M}_n$ in the class $[\tom_{\B\delta}]$ which are equal to the form $\tom_{\B\delta}$ on some neighborhood  of $\Sigma$, and $\mathcal{A}(\B\delta, \rel \Sigma)$ is the space of almost complex structures on  $\widetilde{M}_n$ tamed by a form in $\Omega(\B\delta, \rel \Sigma)$ and which are equal to the standard complex structure $\tilde{J}_0$ on some neighborhood of $\Sigma$. We pass to spaces of structures standard near $\Sigma$ as they are more natural for the study of restriction maps, see Remark~\ref{remark:note on spaces A and inflation} below.

\begin{lemma}\label{L:strata}
Suppose that $\B\delta'>\B\delta$ are two sets of capacities in $\cC_n$ separated by exactly one wall defined by a homology class $A\in \cS(M,\om)^{\leq -1}$. Then,
\[
\mathcal{A}(\B\delta, \rel \Sigma) = \mathcal{A}(\B\delta', \rel \Sigma)\sqcup \cA_k,
\]
where $\cA_k\subseteq \mathcal{A}(\B\delta, \rel \Sigma)$ consists of those almost complex structures $J$ for which the class $A$ is represented by a $J$-holomorphic sphere. Moreover, $\cA_k$ is a Fr\'echet submanifold
of codimension
\begin{equation*}
\label{Eq:codimension}
k = 2 - 2c_1(A).
\end{equation*}
\end{lemma}
\begin{proof}
These results are established in Sections 4.2 and 5.3 of \cite{Anjos-al-Stability}. In their notation, the class $u$ of the symplectic form corresponds to $\B\delta$ (or to $\B{\delta'}$) in our notation. The key observation is that, under the conditions given in the statement, the only negative class in $\cS(M,\om)^{\leq -1}$ whose symplectic area changes sign after crossing the wall is $A$. It follows that $A$ is the only negative class which is represented by embedded $J$-holomorphic spheres for some $J\in \cA(\B\delta,\rel\Sigma)$, but which is not represented by such spheres for any $J\in \cA(\B{\delta'},\rel\Sigma)$. See the proof of~\cite[Theorem 5.1]{Anjos-al-Stability} for a complete discussion in a similar situation.
\end{proof}

In what follows, we use the relation between symplectic blow-ups and embeddings discussed in Appendix~\ref{Appendix}.

Let $\B\delta'>\B\delta$ and let $\Emb(\B\delta',\cp^2)\to \Emb(\B\delta,\cp^2)$ be the map defined by
the restriction of an embedding to smaller balls. By considering the mapping cylinder we get a pair
$(\Emb(\B\delta,\cp^2),\Emb(\B{\delta'},\cp^2))$ of spaces. The following lemma provides a condition
for non-triviality of the corresponding relative homotopy groups.

\begin{lemma}\label{L:eells}
If the codimension $k$ of $\cA_k\subseteq \mathcal{A}(\B\delta, \rel \Sigma)$ is positive then
there exists a non-trivial element in $\pi_k({\Emb}(\B\delta,\cp^2),{\Emb}(\B{\delta'},\cp^2))$.
\end{lemma}
\begin{proof}
It follows from Eells's version of the Alexander-Pontryagin duality
\cite{Eells} that there is an isomorphism 
\begin{equation}
H^0(\cA_k;\Z) \cong H^k(\mathcal{A}(\B\delta, \rel \Sigma),\mathcal{A}(\B\delta', \rel \Sigma);\Z).
\label{Eq:eells}
\end{equation}
Let $1\in H^0(\cA_k;\Z)$ be a generator corresponding to a connected component of $\cA_k$.
Let $u\in H^k(\mathcal{A}(\B\delta, \rel \Sigma),\mathcal{A}(\B\delta', \rel \Sigma);\Z)$ be the image of  $1$
with respect to the above isomorphism. The evaluation of $u$ on a relative homology class
$a\in H_k(\mathcal{A}(\B\delta, \rel \Sigma),\mathcal{A}(\B\delta', \rel \Sigma);\Z)$ is given by the intersection
of the cycle representing $a$ with the connected component of $\cA_k$; see Eells \cite[Page 113]{Eells}. Let $a\in H_k(\mathcal{A}(\B\delta, \rel \Sigma),\mathcal{A}(\B\delta', \rel \Sigma);\Z)$ be represented by a
small disc transverse to $\cA_k$. Since the intersection is a single point, the
class $a$ is non-trivial, and, because it is represented by a disc, it defines a
non-trivial element in $\pi_k(\mathcal{A}(\B\delta, \rel \Sigma),\mathcal{A}(\B\delta', \rel \Sigma))$.

According to \eqref{homotopy commutative diagram} there is the following homotopy commutative diagram of homotopy fibrations 
\begin{equation}
\begin{tikzcd}
\Symp_h(\B\delta',\rel \Sigma) \ar[r]\ar[d] & \Diff_h({\B{\delta'}}, \rel \Sigma)\ar[r]\ar[d,hookrightarrow,"\simeq"] &  \mathcal{A}(\B{\delta'},\rel \Sigma)\ar[d,hookrightarrow] \\
\Symp_h(\B\delta,\rel \Sigma) \ar[r] & \Diff_h({\B{\delta}}, \rel \Sigma)\ar[r] &  \mathcal{A}(\B{\delta},\rel \Sigma),
\end{tikzcd}
\label{Eq:homotopy diagram}
\end{equation}
where $ \Diff_h({\B{\delta}}, \rel \Sigma)$ is the group of diffeomorphisms of the symplectic blow-up of size $\B\delta$ supported away from $\Sigma$, and $\Symp_h(\B\delta, \rel \Sigma)$ is the subgroup of $\Diff_h({\B{\delta}}, \rel \Sigma)$ consisting of diffeomorphisms preserving $\omega_{\B\delta}$. Then, by considering the corresponding commutative diagram of homotopy groups, the standard diagram chase yields  a nontrivial element 
in 
\[
\pi_{k-1}(\Symp_h(\B\delta,\rel \Sigma),\Symp_h(\B\delta',\rel \Sigma)).
\]
Since $\Symp_h(\cp^2)$ acts transitively on $\Emb(\B\delta,\cp^2)$ we get the following commutative diagram of the corresponding fibrations,
\begin{equation}
\begin{tikzcd}
{\rm Stab}_h(\phi(\B\delta'))\ar[r]\ar[d] & \Symp_h(\cp^2) \ar[r]\ar[d,equal] & {\Emb}(\B\delta',\cp^2)\ar[d]\\
{\rm Stab}_h(\phi(\B\delta))\ar[r] & \Symp_h(\cp^2) \ar[r] & {\Emb}(\B\delta,\cp^2),
\end{tikzcd}
\label{Eq:Embs}
\end{equation}
where $\phi\colon B^4(\B\delta')\to \cp^2$ is a symplectic embedding and the entries of the leftmost column are the isotropy subgroups. Moreover, by Proposition~\ref{prop:equivalence stab and rel B} and the blow-up construction, there is a homotopy equivalence
\[
\Symp_h(\B\delta,\rel \Sigma)\;\simeq\;{\rm Stab}_h(\phi(\B\delta)),
\]
and similarly for \(\B\delta'\). In Proposition \ref{prop:induced map between Symp groups}, it is shown that 
\[
\Symp_h(\B\delta',\rel \Sigma)\longrightarrow \Symp_h(\B\delta,\rel \Sigma),
\]
induced by the inclusion 
$\mathcal{A}(\B\delta',\rel \Sigma)\hookrightarrow \mathcal{A}(\B\delta,\rel \Sigma)$
in diagram~\eqref{Eq:homotopy diagram}, is compatible with the inclusion of isotropy subgroups
\({\rm Stab}_h(\phi(\B\delta'))\hookrightarrow {\rm Stab}_h(\phi(\B\delta))\).
Hence, we obtain a non-trivial class in
\[
\pi_{k-1}\bigl({\rm Stab}_h(\phi(\B\delta)),{\rm Stab}_h(\phi(\B\delta'))\bigr).
\]
Applying once again a diagram chase to the long exact sequences of homotopy groups associated with
diagram~\eqref{Eq:Embs}, we obtain a non-trivial class in
\[
\pi_k\bigl(\Emb(\B\delta,\cp^2),\Emb(\B\delta',\cp^2)\bigr),
\]
which completes the proof.
\end{proof}

\begin{proposition}\label{prop:maximality of small chamber}
Let $\B{\delta'},\B\delta \in \cC_n$ be such that
$\B{\delta'}> \B\delta$ and suppose $\B\delta$ is separated from
$\B{\delta'}$ by exactly one wall. Then the restriction map 
$\Emb(\B{\delta'},\cp^2)\to \Emb(\B\delta,\cp^2)$ is not a homotopy equivalence. In particular, if $\B{\delta_0} \in \Delta_0$ and $ \B\delta \in \cC_n$ satisfy $\B{\delta} > \B{\delta_0}$ and $\B\delta$ is separated from
$\B{\delta_0}$ by exactly one wall, then the map $\Emb(\B{\delta},\cp^2)\to \Emb(\B{\delta_0},\cp^2)$ is not 
a homotopy equivalence.
\end{proposition}
\begin{proof}
By Lemma~\ref{L:strata} we have $\mathcal{A}(\B{\delta},\rel \Sigma) = \mathcal{A}(\B{\delta'},\rel \Sigma)\sqcup \cA_k$, where \(\cA_k\) denotes the stratum consisting of almost complex structures for which the wall class \(A\) is represented by a \(J\)-holomorphic sphere. Moreover, this stratum has codimension $k= 2 -2c_1(A)$. It then follows from Lemma~\ref{L:eells} that
\[
\pi_{2-2c_1(A)}
\bigl(\Emb(\B{\delta},\cp^2),\Emb(\B{\delta'},\cp^2)\bigr)\neq 0,
\]
and hence the restriction map
\[
\Emb(\B{\delta'},\cp^2)\longrightarrow \Emb(\B{\delta},\cp^2)
\]
is not a homotopy equivalence. In particular, if the unique wall separating  $\B{\delta_0}$ and $\B{\delta}$ is the wall defined by \(L-(E_1+\cdots+E_n)\), then \(k=2n-4\). 
\end{proof}

\begin{remark}\label{remark:note on spaces A and inflation}
There are two main reasons why it is more convenient to work with the spaces \(\mathcal{A}(\B\delta,\rel \Sigma)\) rather than with the spaces
\(\cA(\B\delta, \Sigma)\) used in the proof of Theorem~\ref{T:main}.
On the one hand, the stabilizers are exactly \(\Symp_h(\B\delta,\rel \Sigma)\), which leads to a simpler analysis of the induced maps between stabilizers and of the homotopy commutativity of the diagram associated with restriction to smaller capacities. On the other hand, the \(J\)-tame inflation process along \(\Sigma\), used to establish the inclusions between spaces of almost complex structures, can be carried out in such a way that the property of \(J\) being standard near \(\Sigma\) is preserved. Both features are essential for our proof of homotopy commutativity that uses the definition of symplectic blow-up as a surgery (see Proposition~\ref{prop:induced map between Symp groups}).

Note that this setting is well suited only for restriction arguments. If one instead wishes to increase the capacities, as in Theorem~\ref{T:main}, one must perform inflation along curves that intersect \(\Sigma\) positively. This process cannot be carried out while preserving the condition that the almost complex structure is standard near \(\Sigma\).\eoe
\end{remark}

\section{Infinitely many \texorpdfstring{$(-2)$}{-2} walls for 
\texorpdfstring{$n= 9$}{n =9} balls} \label{S:1/3}

Note that the Light Cone Lemma \cite[Lemma 2.6]{LiLiu-Adjunction} implies that homology classes $A \in H_2(M;\Z)$ represented by $J$-holomorphic curves, and which may define walls in the set of admissible capacities, must have negative self-intersection. Therefore, the first step is to identify such classes.  

Let $J$ be a generic $\omega_{\B{\delta}}$-tame almost complex structure on $M_{\B{\delta}}$ for which every exceptional class $E_i$, for $i=1,\ldots,9$, is represented by an embedded $J$-holomorphic curve. The existence of such a $J$ follows from the blow-up construction. 

Recall that if $A \in H_2(M;\Z)$ is a homology class such that
\[
k(A) = \frac{1}{2}\big(A \cdot A + c_1(A)\big) \geq 0,
\]
then the Gromov invariant of $A$, defined by Taubes in \cite{Taubes-Counting}, counts, for a generic almost complex structure $J$ tamed by $\omega$, the algebraic number of embedded $J$-holomorphic curves in class $A$ passing through $k(A)$ generic points. Taubes’s curves arise as zero sets of sections and hence need not be connected. Moreover, components of non-negative self-intersection are embedded holomorphic curves of some genus $g \geq 0$, while all other components are exceptional spheres. 

It follows from Gromov’s Compactness Theorem~\cite{McD-S-JHolomorphicCurves} that, given a tamed $J$, any class $A$ with non-zero Gromov invariant $\rm Gr(A)$ is represented by a collection of $J$-holomorphic curves or cusp curves. For a rational surface with canonical class $K_n = c_1(T^*M_{\B\delta})$, since $b_2^+ = 1$, it follows from the wall-crossing formula of Li and Liu for Seiberg–Witten invariants \cite[Corollary 1.4]{LiLiu-WallCrossing} that for every class $A$ satisfying $k(A) \geq 0$ we have
\[
\left|{\rm Gr}(A) - {\rm Gr}(K_n - A)\right| = 1.
\]
In our case, consider the class $D_9 = 3L - (E_1 + \cdots + E_9)$. Then
\[
k(D_9) = \tfrac{1}{2}\big(c_1(D_9) + D_9 \cdot D_9\big) = D_9 \cdot D_9 = 0,
\]
and it follows that
\[
\big|{\rm Gr}(D_9) - {\rm Gr}(K_9 - D_9)\big| = 1.
\]
The symplectic area of $D_9$ is given by
\begin{equation}
\omega_{\B{\delta}}(D_9) = 3 - \sum_{i=1}^9 \delta_i.
\label{Eq:omegaD9}
\end{equation}
If the sum of the capacities is smaller than $3$ then $\omega_{\mathbf{\delta}}(D_9)>0$. As $K_9 = -D_9$, we have $\omega_{\B\delta}(K_9 - D_9) = -2\omega_{\B\delta}(D_9) < 0$, which implies that $\mathrm{Gr}(K_9 - D_9) = 0$ and therefore $\mathrm{Gr}(D_9) = \pm 1$. Hence, the class $D_9$ is represented by either an embedded torus, a sphere with a simple double point or an ordinary cusp, or a reducible cusp-curve. Since $\mathrm{Gr}(D_9)$ depends only on the deformation class of the symplectic form, it follows that $\mathrm{Gr}(D_9) = \pm 1$ for all symplectic forms $\omega_{\mathbf{\delta}}$.

Let $A=a_0L-(a_1E_1+\cdots+a_9E_9)\in H_2(M;\Z)$ be represented by a 
$J$-holomorphic curve $u\colon \Sigma\to M$. Since each $E_i$ is represented
by a simple $J$-curve, it follows from positivity of intersections that
\begin{equation}
0\leq E_i\cdot A = a_i, 
\label{Eq:ai>0}
\end{equation}
for $i=1,\ldots,9$.
Moreover, by positivity of symplectic area, for some $\B{\delta}$ we have
\[
0< \omega_{\B{\delta}}(A) = a_0 - \sum_{i=1}^9 \delta_ia_i,
\]
which implies $a_0>0$. Furthermore, observe that
\[
c_1(A)=A\cdot D_9 = 3a_0 - \sum_{i=1}^9 a_i
\]
which implies
\begin{align}
3a_0 &= c_1(A) + \sum_{i=1}^9 a_i.
\label{}
\end{align}
The positivity of intersection implies that either $c_1(A)=A\cdot D_9\geq 0$
or $c_1(A)=A\cdot D_9<0$ and $A$ is represented by a $J$-holomorphic curve
$u\colon \Sigma\to M$ which is a component of the cusp-curve representing
$D_9$. Both cases break into several subcases, which are described in the following lemma. 

\begin{lemma}\label{lemma:NegativeClasses_9balls}
Let $J$ be an almost complex structure on $M_9$ which is tamed by some symplectic form $\omega_{\B{\delta}}$ and for which each of the exceptional classes $E_1,\ldots,E_9$ is represented by an embedded $J$-holomorphic sphere. Suppose $A \in H_2(M_9, \Z)$ is represented by a $J$-holomorphic curve with negative self-intersection $A\cdot A < 0$. Then $A$ is of one of the following forms. 
\begin{enumerate}
\item If $c_1(A)< 0$ then 
\begin{enumerate}
\item $A = L - \sum_{i\in I} E_i$, with $|I|\geq 4$;
\item $A = 2L - \sum_{i\in I} E_i$, with $|I|\geq 7$; 
\item $A = 3L - 2E_i - \sum_{j\neq i} E_j$.
\end{enumerate}
\item If $c_1(A)=0$ then
\begin{enumerate}
\item $A = L-E_i-E_j-E_k + mD_9$ for pairwise distinct $i,j,k$ and $m\geq 0$;
\item $A = -(L-E_i-E_j-E_k) + mD_9$ for pairwise distinct $i,j,k$ and $m\geq 1$;
\item $A = \pm(E_i-E_j) + mD_9$ for $m\geq 1$;
\end{enumerate}
\item If $c_1(A)>0$ then $A$ is an exceptional class. 
\end{enumerate}
\end{lemma}

\begin{proof}
\begin{enumerate}
\item If  $c_1(A) = A\cdot D_9 < 0$ then the $J$-curve $u\colon \Sigma\to M_{\B{\delta}}$ representing $A$ is a multiple of a simple component of a cusp-curve 
\[
C = m_1C_1+\cdots + m_kC_k
\]
representing $D_9$, where the components $C_i$ are simple $J$-holomorphic curves and where $m_i>0$. Since $D_9 = 3L - (E_1+\cdots+E_9)$ it follows that $A$ is of the form $A= a_0L - \sum_{i=1}^9 a_i E_i$, where $a_0=1,2$ or $3$. If $a_0=1$, then the adjunction inequality implies that 
\[
1 + \frac12 (A \cdot A -c_1(A) )
= 1 + \frac12 \left(1 - \sum_{i=1}^9a_i^2  - 3 + \sum_{i=1}^9 a_i\right) \geq 0.
\] 
Hence 
\[
\sum_{i=1}^9a_i(1-a_i) \geq 0,
\]
which implies that $a_i \in \{0,1 \}$. Since we have $c_1(A)= 3 - \sum_{i=1}^9 a_i < 0$, it follows that 
\[
A = L - \sum_{i\in I} E_i,
\]
where $|I|\geq 4$. This proves case (1.a). Cases (1.b) and (1.c) are obtained in a similar way.
\item If $c_1(A)=0$, then the adjunction inequality implies that $A\cdot A=-2$, since the case $A\cdot A=-1$ would make $A$ an exceptional class, which does not define an interior wall.  Classes satisfying  $A\cdot A=-2$ are called  root classes. Since $n=9$, it follows from \cite[Remark 1]{Demazure-PointsCP^2} that $A$ is a sum of simple roots, each of which is of the form $\pm(L-E_i-E_j-E_k)$ or $\pm(E_i-E_j)$, plus a multiple of $D_9$.  This proves (2). Note that in cases (2.b)  and (2.c) the integer $m\in \Z$ is positive by positivity of intersections, as $-(L-E_i-E_j-E_k)\cdot E_i=-1$ and $(E_i-E_j)\cdot E_i=-1$.
\item If $ c_1(A) >0$ and $A\cdot A<0$, then the adjunction formula implies that $c_1(A)=1$. More precisely, if $c_1(A)\geq 2$ then the adjunction formula implies that $A\cdot A\geq 0$. Moreover, if $c_1(A)=1$ then the adjunction formula gives 
\[
A\cdot A \geq  2 g(\Sigma) -1 \geq -1,
\]
because $ g(\Sigma)\geq 0$. Hence $A\cdot A=-1$ and $A$ is an embedded symplectic exceptional class.
\end{enumerate}
\end{proof}

\begin{remark}
The enumeration of classes in Lemma~\ref{lemma:NegativeClasses_9balls} is
closely related to sets of homology classes arising in several classical
problems. Restricting to classes of self-intersection $A^2 \geq -2$ recovers
the root classes of the lattice $K_9^\perp \subset \operatorname{Pic}(\widetilde{M}_9)$. For $n \leq 8$ these are the $(-2)$-classes whose effectivity governs the degenerations of weak del~Pezzo surfaces, classified in the language of the Kantor--Cremona group by Coble~\cite{Coble} and Du~Val~\cite{DuVal}; see also~\cite[Ch.~8]{DolgachevCAG} for a modern treatment and \cite{Demazure-PointsCP^2} for the del~Pezzo formalism. For $n\leq 10$, the purely linear part of this combinatorics (i.e. the collinearity relations among the points) is governed by the rank-$3$ matroid of linear dependences of the configuration; for an explicit treatment of the corresponding incidence strata see~\cite{IKKLPW}. However, the full list of classes relevant to the embedding problem treated here (which includes classes of self-intersection below $-2$) does not seem to appear in any classical reference.\eoe
\end{remark}

We know from Lemma~\ref{lemma:NegativeClasses_9balls} that the only family of classes that may define infinitely many walls intersecting the interior of the set $\cC_9$ of admissible capacities are the following $(-2)$-classes:
\begin{itemize}
\item $A_m(+,i,j,k) := L-E_i-E_j-E_k + mD_9$ for pairwise distinct $i,j,k$ and $m\geq 0$,
\item $A_m(-,i,j,k) := -(L-E_i-E_j-E_k) + mD_9$ for pairwise distinct $i,j,k$ and $m\geq 1$, 
\item $A_m(\pm,i,j) := \pm(E_i-E_j) + mD_9$ for pairwise distinct $i,j$ and $m\geq 1$.
\end{itemize}
We now briefly recall some facts about these $(-2)$-classes. For more details, the reader can consult~\cite[Section~2.4]{Anjos-al-Stability}.

Given a $(-2)$-class $A$, let $r_A$ denote the reflection along the hyperplane $\ell_A=0$ in $H_2(M_9,\Z)$, that is,
\[
r_A(B) = B -2 \frac{(A\cdot B)}{(A\cdot A)} A =B + (A\cdot B)A.
\]
It is known that the reflections about the classes $A_m$ defined above, together with the ones with $m\leq 0$, generate the group $\Aut_{c_1}$ of linear transformations of the lattice $H_2(M_9,\Z)$ preserving the intersection product and the class $D_9$. It is also known that this group acts on the symplectic cone of $M_9$ and that a fundamental domain for this action is the set of so-called \emph{reduced classes}
\[
C_{\text{red}} := \left\{a_0L-\sum a_iE_i ~|~ a_0\geq a_1+a_2+a_3,~ a_1\geq\cdots\geq a_9>0\right\}.
\]
Let us call this set the reduced chamber. The images of the region $C_{\text{red}}$ under the action of $\Aut_{c_1}$ tessellate the symplectic cone.

Define the \emph{elementary roots} as $r_{ijk}=L-E_i-E_j-E_k$ and $r_{ij}=E_i-E_j$ for distinct indices $i,j,k$. The set $\cR_9$ of all roots is given by 
\[
\cR_9=\{ \pm r_{ijk}+mD_9~|~m\in\Z\} \cup \{ \pm r_{ij}+mD_9~|~m\in\Z\}.
\]
Let $\cR_9^+$ be the set of roots defining $(-2)$-walls in $\cC_9$. By Lemma \ref{lemma:NegativeClasses_9balls}, the elements of   $\cR_9^+$ are precisely the \emph{positive roots} in the sense of~\cite{Looijenga-81}; that is, they are  roots of the form
\begin{align} 
r_{ijk} +mD_9,~&~ m\geq 0,\nonumber\\ 
-r_{ijk}+mD_9,~&~ m\geq 1, \label{Eq:relevant}\\ 
\pm r_{ij}+mD_9,~&~ m\geq 1. \nonumber
\end{align}
The walls of the reduced chamber are defined by the \emph{simple roots}
\begin{align*}
r_0=r_{123}&:=L-E_1-E_2-E_3\\
r_1=r_{12}&:=E_1-E_2\\
&\vdots\\
r_8=r_{89}&:=E_8-E_9
\end{align*}
together with the exceptional class $E_9$. Since 
\[
D_9=3r_0+2r_1+4r_2+6r_3+5r_4+4r_5 + 3r_6+2r_7+r_8
\]
it follows that we can express any positive $(-2)$-class $\alpha$ as a linear combination $\alpha=\sum_i a_i r_i$ of the simple roots in which the coefficient $a_0$ is strictly positive.

The intersection graph of the $(-2)$-walls defining the reduced chamber $C_{\text{red}}$ is the $\widetilde{E_8}$ diagram
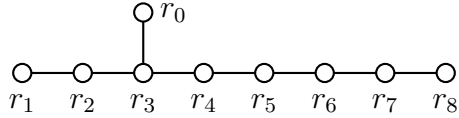
\begin{center}
\begin{tikzpicture}[scale=.4]
    \foreach \x in {0,...,7}
    \draw[thick,xshift=\x cm] (\x cm,0) circle (3 mm);
    \foreach \y in {0,...,6}
    \draw[thick,xshift=\y cm] (\y cm,0) ++(.3 cm, 0) -- +(14 mm,0);
    \draw[thick] (4 cm,2 cm) circle (3 mm);
    \draw[thick] (4 cm, 3mm) -- +(0, 1.4 cm);
    
    \foreach \x in {3,...,8}
    \draw[xshift=-2 cm] (2*\x cm, -1) node {$r_{\x}$};
    
    \draw (5,2) node {$r_0$};
    \draw (2,-1) node {$r_2$};
    \draw (0,-1) node {$r_1$};  
\end{tikzpicture}
\end{center}
Any other $(-2)$-chamber is the image of $C_{\text{red}}$ under a Cremona transformation; consequently, its $(-2)$-walls intersect in a similar way. The only positive wall of $C_{\text{red}}$ is the one defined by the simple root $r_0$. In this setting, $C_{\text{red}}$ is a Weyl chamber of the root system $\widetilde{E_8}$.

Note that the region $\Delta_0$ in the statement of Theorem~\ref{T:sequence} imposes no ordering while the reduced chamber $C_{\mathrm{red}}$ constrains only the sizes of the three largest capacities. One has
\[
\Delta_0\cap\{\delta_1\geq\cdots\geq\delta_9\} = C_{\mathrm{red}}\cap\{\delta_1+\cdots+\delta_9<1\}.
\]

\subsection{Proof of Theorem \ref{T:sequence} parts (1) and (2)}

In this proof we use the identification of $\cC_9$ with a subset of $H_2(\widetilde{M}_9;\R)$ via the affine embedding $\B\delta\mapsto {\rm PD}[\omega_{\B\delta}]$.

Starting with $C_0 = \Delta_0$, the goal is to find $C_1, C_2,  \ldots, C_k, \ldots$ as consecutive reflections with respect to $A_1, A_2, \ldots$ where $A_i$ defines a positive $(-2)$ wall of the chamber $r_{A_{i-1}} \cdots r_{A_1}(C_{\text{red}})$. Then, given $\B\delta_0 \in C_0\cap C_{\text{red}}$, we obtain a sequence $\B\delta_k \in C_k$. The existence of such walls follows from the following lemma.

\begin{lemma}
Finding a sequence of positive $(-2)$ walls is equivalent to playing Mozes' Game of Numbers on the affine Dynkin diagram $\widetilde{E_8}$.
\end{lemma}
\begin{proof}
The mathematical game invented by S. Mozes in \cite{Mozes} assigns real numbers to the nodes of any connected graph. A move consists of selecting a node with a positive number, adding this number to each adjacent node, and then reversing its sign at the selected node. 

Observe that reflecting a $(-2)$-chamber across a positive wall corresponds to a move in Mozes' game. Indeed, since we consider only reflections across positive walls, it suffices to track the coefficient $a_0$ of $r_0$ in the $(-2)$-class associated with each node. By the reflection formula, performing such a reflection adds the coefficient $a_0$ of the positive wall to the adjacent nodes, because the corresponding classes have intersection number $1$. The coefficient of the positive wall itself changes sign, since $r_A(A)=-A$, and all other nodes are left unchanged, since $r_A(B)=B$ whenever $A\cdot B=0$.
\end{proof}

\begin{proposition}
There is a sequence $\{\B\delta_k\} $ in the symplectic cone of $M_9$ such that  $\B\delta_0\in \Delta_0\cap C_{\text{red}}$ and any two consecutive terms $\B\delta_k$ and $\B\delta_{k+1}$ are separated by exactly one $(-2)$-wall $W_A$ such that  $\omega_{\B\delta_k}(A)>0$ and $\omega_{\B\delta_{k+1}}(A)<0$.
Moreover, $\lim_{k\to \infty}\B\delta_k = (1/3,1/3,\ldots,1/3)$.
\end{proposition}
\begin{proof}
We begin by constructing an infinite sequence $\{\B{\delta_k'}\}$ in the symplectic cone that crosses only positive walls, starting from the reduced chamber $C_{\text{red}}$. This implies that we always play Mozes' game on a node with a positive integer. The starting graph, as explained above, is:
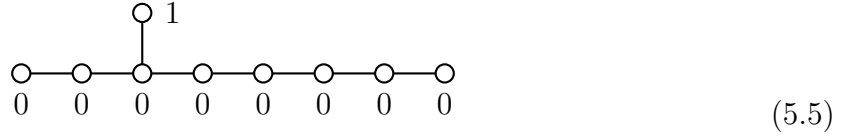
\begin{equation}\label{fig:first_graph}
\begin{tikzpicture}[scale=.4]
    \foreach \x in {0,...,7}
    \draw[thick,xshift=\x cm] (\x cm,0) circle (3 mm);
    \foreach \y in {0,...,6}
    \draw[thick,xshift=\y cm] (\y cm,0) ++(.3 cm, 0) -- +(14 mm,0);
    \draw[thick] (4 cm,2 cm) circle (3 mm);
    \draw[thick] (4 cm, 3mm) -- +(0, 1.4 cm);
    
    \foreach \x in {3,...,8}
    \draw[xshift=-2 cm] (2*\x cm, -1) node {$0$};
    
    \draw (5,2) node {$1$};
    \draw (2,-1) node {$0$};
    \draw (0,-1) node {$0$};  
\end{tikzpicture}
\end{equation}
In order to guarantee that there always exists a node with a positive integer, we use Lemma 3 from \cite{Eriksson}, which ensures the existence of an invariant of the game. This invariant is given by the standard inner product of the vector $(x_0, x_1, \ldots, x_8)$, representing the node values at any moment of the game, with the vector $(3, 2, 4, 6, 5, 4, 3, 2, 1)$, which corresponds to the particular graph \eqref{fig:looper}. The starting graph \eqref{fig:first_graph} implies that the invariant of the game in our case is 3, which in turn guarantees that there is always at least one node with a positive value. Therefore, we can always perform a reflection across a positive wall. Moreover, condition (iii) on page 165 of \cite{Eriksson} implies that the game is infinite and non-periodic, since the inner product between $(3, 2, 4, 6, 5, 4, 3, 2, 1)$ and the initial position $(1,0,0,0,0,0,0,0,0)$ is positive. 

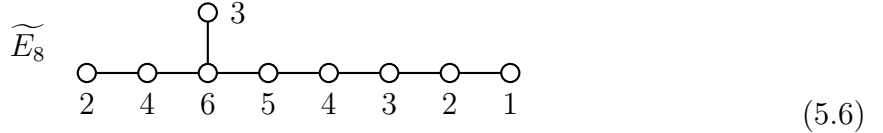
\begin{equation}\label{fig:looper}
  \begin{tikzpicture}[scale=.4]
    \draw (-1,1) node[anchor=east]  {$\widetilde{E_8}$};
    \foreach \x in {0,...,7}
    \draw[thick,xshift=\x cm] (\x cm,0) circle (3 mm);
    \foreach \y in {0,...,6}
    \draw[thick,xshift=\y cm] (\y cm,0) ++(.3 cm, 0) -- +(14 mm,0);
    \draw[thick] (4 cm,2 cm) circle (3 mm);
    \draw[thick] (4 cm, 3mm) -- +(0, 1.4 cm);
    
    \foreach \x in {1,...,6}
    \draw[xshift=16 cm] (-2*\x cm, -1) node {$\x$};
    
    \draw (5,2) node {$3$};
    \draw (2,-1) node {$4$};
    \draw (0,-1) node {$2$};    
  \end{tikzpicture}
\end{equation}
Since 
\[
[\omega_{\B {\delta_{k+1}'}}] = [\omega_{\B {\delta_{k}'}}] + [\omega_{\B {\delta_{k}'}}] (A_{k+1}) {\rm PD}[A_{k+1}],
\] 
it follows that $\B {\delta_{k+1}'} > \B {\delta_{k}'}$, because $[\omega_{\B {\delta_{k}'}}] (A_{k+1}) > 0$ and $A_{k+1}$ has non-negative coefficients in the basis $\{L, -E_1,  \ldots, -E_9 \}$. 

Note that the sequence $\{\B{\delta_k'}\}$ is  not normalized. 
We  rescale $\{\B{\delta_k'}\}$ to get a symplectic class 
\[
[\omega_{\B {\delta_k}}]=  L - \sum \delta_{k,i} E_i  = \frac{[\om_{\B {\delta_k'}}]}{[\om_{\B {\delta_k'}}](L)}.
\]
Therefore
\begin{equation*}
[\omega_{\B {\delta_{k+1}}}] (E_l)  = \frac{[\om_{\B {\delta_{k+1}'}}](E_l)}{[\om_{\B {\delta_{k+1}'}}](L)} = \frac{[\om_{\B {\delta_{k}'}}](E_l) + [\om_{\B {\delta_{k}'}}](A)  (A \cdot E_l)}{[\om_{\B {\delta_{k}'}}](L) + [\om_{\B {\delta_{k}'}}](A) ( A \cdot  L)}
\end{equation*}
where $A$ is one of the positive classes \eqref{Eq:relevant}. If $A= \pm r_{ijk} +mD_9$ then we obtain 
\begin{equation}\label{Eq:normalized1}
[\omega_{\B {\delta_{k+1}}}] (E_l)  = \left\{
\begin{array}{l}
\displaystyle{\frac{[\om_{\B {\delta_{k}'}}](E_l) + (m \pm 1)[\om_{\B {\delta_{k}'}}](A)}{[\om_{\B {\delta_{k}'}}](L) + (3m \pm 1)[\om_{\B {\delta_{k}'}}](A)}} \quad \mbox{if} \ l=i, j, k  \\ \\
 \displaystyle{\frac{[\om_{\B {\delta_{k}'}}](E_l) + m[\om_{\B {\delta_{k}'}}](A)}{[\om_{\B {\delta_{k}'}}](L) + (3m \pm 1 )[\om_{\B {\delta_{k}'}}](A)}} \quad \mbox{if} \ l \neq  i, j,  k 
\end{array} \right. 
\end{equation}
If  $A= \pm r_{ij} + mD_9$  then it follows that
\begin{equation}\label{Eq:normalized2}
[\om_{\B {\delta_{k+1}}}] (E_l)  = \left\{
\begin{array}{l}
\displaystyle{\frac{[\om_{\B {\delta_{k}'}}](E_l) + (m \mp 1)[\om_{\B {\delta_{k}'}}](A)}{[\om_{\B {\delta_{k}'}}](L) + 3m[\om_{\B {\delta_{k}'}}](A)}} \quad \mbox{if} \ l=i, \\ \\
\displaystyle{\frac{[\om_{\B {\delta_{k}'}}](E_l) + (m \pm 1)[\om_{\B {\delta_{k}'}}](A)}{[\om_{\B {\delta_{k}'}}](L) + 3m[\om_{\B {\delta_{k}'}}](A)}} \quad \mbox{if} \ l=j, \\ \\ 
\displaystyle{\frac{[\om_{\B {\delta_{k}'}}](E_l) + m[\om_{\B {\delta_{k}'}}](A)}{[\om_{\B {\delta_{k}'}}](L) + 3m[\om_{\B {\delta_{k}'}}](A)}} \quad \mbox{if} \ l \neq  i, j. 
\end{array} \right. 
\end{equation}

Finally, since we are crossing an infinite number of walls, it follows that $m \to \infty$. Therefore, \eqref{Eq:normalized1} and \eqref{Eq:normalized2} imply that $\lim_{k \to \infty} \B\delta_k = (1/3, 1/3, \ldots, 1/3)$.
\end{proof}
This concludes the proof of parts (1) and (2) in Theorem~\ref{T:sequence}.

\begin{remark}
Since the sequence of chambers $\{C_k\}$ accumulates at the monotone point $\B\delta_\infty = (1/3, 1/3, \ldots, 1/3)$, the limit of the blow-up forms $\tom_{\B\delta_k}$ is the dual of the first Chern class $c_1=L-\frac{1}{3}\sum E_i$. As the classes in $\cS^{\leq -3}$ have first Chern class at most $-1$, it follows that their symplectic area with respect to the forms $\tom_{\B\delta_k}$ is eventually negative. In particular, after finitely many steps, the sequence $\{\B\delta_k\}$ crosses only $(-2)$ walls. More generally, one can show that the monotone point $\B\delta_\infty$ is the only accumulation point of the Weyl chambers in the set $\cC_9$ of admissible capacities. Consequently, the walls defined by negative classes of self-intersection $\leq -3$ intersect $\cC_9$ in a relatively compact set containing $\Delta_0$.\eoe
\end{remark}

\subsection{Proof of Theorem~\ref{T:sequence} part (3)}
Note that after renormalization, the capacities in the sequence obtained from Mozes' game are not necessarily pairwise comparable, that is, we may not have $\B\delta_k < \B{\delta}_{k+1}$. To prove statement (3) of Theorem \ref{T:sequence}, we first show that given two consecutive chambers as constructed above, we can find comparable capacities $\B\delta \in C_k$ and $\B{\delta'}\in C_{k+1}$ such that $\B\delta < \B{\delta'}$. For this, we introduce the following purely algebraic definition and lemma.

\begin{definition}
Let $P\subset \R^m$ be the closure of the positive orthant (with the standard inner product). Let $n$ be a vector of $\R^m$ and let $W(n,c)=\{x~|~x\cdot n = c\}$ be an affine hyperplane that divides $\R^m$ into two chambers $H^+_n$ and $H^-_n$. We say that there is a restriction from $H^+_n$ to $H^-_n$ if we can find two points $x^{\pm}\in H^{\pm}_n$ such that $x^+ - x^- \in P$.
\end{definition}

Observe that for nonzero $x,y\in P$, $x\cdot y\geq 0$.

\begin{lemma}\label{lemma:solstice lemma}
If $n\in P$, then there is a restriction from $H^+_n$ to $H^-_n$, but no restriction from $H^-_n$ to $H^+_n$. Conversely, if $n\in -P$, then there is a restriction from $H^-_n$ to $H^+_n$, but no restriction from $H^+_n$ to $H^-_n$. Finally, if $n\not\in \pm P$, then there is a restriction from $H^+_n$ to $H^-_n$ and there is also a restriction from $H^-_n$ to $H^+_n$.
\end{lemma}
\begin{proof}
Suppose $n\in P$. Take $x^-\in H^-_n$ and let $x^+\in H^+_n$ be its reflection along $W(n,c)$. By definition, $x^+-x^- = kn$ for some $k>0$, so that $x^+-x^-\in P$. Conversely, starting with $x^+\in H^+_n$, any element in $H^-_n$ can be written as a linear combination $x^-= (x^+ - kn)+w$ for $k>0$ and $w\in W(n,0)$. Therefore, $x^- - x^+ = w-kn$, so that $(x^- - x^+)\cdot n = -k(n\cdot n)<0$. Consequently, $x^- - x^+\not\in P$. This proves the first assertion.\\

The second assertion follows immediately from the first by considering $-n$ instead of $n$.\\

For the third assertion, suppose $n\not\in\pm P$, and pick $x^+\in H^+_n$. Any element $x^-\in H_n^-$ can be expressed as $x^-= (x^+ - kn)+w$ for $k>0$ and $w\in W(n,0)$, so that $x^+ - x^- = kn-w$. Because $W(n,0)$ intersects the interior of $P$, by choosing $w$ appropriately, we can achieve $x^+ - x^-\in P$. Conversely, if we pick some $x^-\in H_n^-$, every $x^+\in H^+_n$ can be written as $x^+=x^- +kn+w$ for $k>0$ and $w\in W(n,0)$, and
\[x^- - x^+ = -kn-w.\]
As before, we can find $w$ so that $x^- - x^+\in P$.
\end{proof}

Let $C_k$ and $C_{k+1}$ be two consecutive chambers separated by a wall defined by a $(-2)$ class $A_k$ as defined above. We can view the chambers as subsets of $\R^9$. Similarly, the class $A_k=a_0 L-\sum a_iE_i$ can be seen as an element of $\R^9$ by only keeping the coefficients $(a_1,\ldots,a_9)$. Positivity of the root $A_k$ puts it in $P$. We can then apply Lemma~\ref{lemma:solstice lemma} with $H^-=C_k$, $H^+=C_{k+1}$, and $n=A_k$ to conclude that there exists a restriction from $C_{k+1}$ to $C_k$ defined by vectors $\B{\delta'}=x^+$ and $\B\delta=x^-$. In particular, $\B{\delta'} > \B\delta$. Moreover, by taking these sets of capacities arbitrarily close to the wall defined by $A_k$, we can ensure that no other $(-2)$ wall lies between $\B{\delta'}$ and $\B\delta$.

Notice that, according to Lemma~\ref{lemma:NegativeClasses_9balls}, there are only finitely many classes of self-intersection at most $-3$. It follows that any two capacities $\B\delta \in C_k$ and $\B{\delta'} \in C_{k+1}$ satisfying $\B\delta < \B{\delta'}$ are separated by finitely many walls among which there is exactly one wall corresponding to a class $A_{k}$ of self-intersection $-2$ and all other walls correspond to classes of self-intersection at most $-3$. Let $\B\gamma_0=\B\delta<\B\gamma_1 <\ldots<\B\gamma_{\ell}=\B{\delta'}$ be a sequence of capacities such that $\B\gamma_i$ and $\B\gamma_{i+1}$ are separated by exactly one wall. If this wall corresponds to a class of self-intersection at most $-3$ then 
\[
\cA(\B\gamma_i,\rel \Sigma) = \cA(\B\gamma_{i+1},\rel \Sigma)\sqcup \cA_i,
\]
where $\cA_i$ is a submanifold of codimension at least $4$. It follows that the inclusion
\[
\cA(\B\gamma_{i+1},\rel \Sigma)\into \cA(\B\gamma_i,\rel \Sigma)
\]
induces isomorphisms on $\pi_k$ 
for $k\leq 2$.
If the wall corresponds to a $(-2)$-class then
$\pi_2(\cA(\B\gamma_{i},\rel \Sigma),\cA(\B\gamma_{i+1},\rel \Sigma)) \neq 0$,
according to the proof of Lemma~\ref{L:eells}. It follows that
\[
\pi_2\big(\cA(\B\delta,\rel \Sigma),\cA(\B{\delta'},\rel  \Sigma)\big) \neq 0,
\]
hence, by Lemma~\ref{L:eells}, the restriction map 
\[
\Emb(\B {\delta'},\cp^2)\to \Emb(\B{\delta},\cp^2)
\]
is not a homotopy equivalence as claimed. This finishes the proof of Theorem~\ref{T:sequence}.

\subsection{Accumulation of walls for \texorpdfstring{$n\geq 10$}{n=>10}}
Let $n=9+k$, where $k>0$ and  $D_n=3L - (E_1+\cdots + E_n).$ Let $I\subseteq \{1,\ldots,n\}$ be a subset of cardinality $k$ and let
\[
D_9(I) = 3L - \sum_{i\notin I}E_i.
\]
It follows that
\[
D_n = D_9(I) - \sum_{i\in I}E_i.
\]
The same argument as in the case for $n=9$ balls shows that ${\rm Gr}(D_9(I))=1$ and hence $D_9(I)$ is also represented by either an embedded torus or a cusp curve.

For all $n\geq 10$, although we do not know whether stability holds outside $\Delta_0$, it follows from the symplectic blow-up construction that there are infinitely many $(-2)$-walls intersecting the interior of the set $\cC_n$ of admissible capacities, accumulating near the boundary points corresponding to the classes $\frac{1}{3}D_9(I)$. As we will now show, accumulation of walls can only occurs at boundary classes $\B\delta$ where the symplectic volume of $(\widetilde{M}_{\B{\delta}}, \om_{\B{\delta}})$ vanishes.

Let $\mathcal{H}_n\subset H_2(\widetilde{M}_n,\Z)$ be the set of classes $A=a_0L-\sum_{i=1}^n a_iE_i$ such that 
\begin{itemize}
\item $a_i\geq 0$, for $i=0,\ldots n$, 
\item $A^2=a_0^2-\sum_i a_i^2<0$, 
\item and satisfying the adjunction inequality $-2\leq A^2 + K_n\cdot A$ 
\end{itemize}
where $K_n$ is the canonical class. The set $\cH_n$ contains the homology classes of all possible $J$-holomorphic curves (of any genus) that may obstruct the $J$-tame inflation procedure on $\widetilde{M}_n$ endowed with any symplectic form $\om_{\B{\delta}}$. Given $A\in\cH_n$, let $H_A$ be the associated wall in $\cC_n$, that is, the kernel of the functional $\ell_A(\om):=\om(A)$.

Let $\cP^+_n=\{\om~|~\om^2>0~\text{and}~\om\cdot L >0\}$ be the forward positive cone in $H_2(\widetilde{M}_n,\R)$ and let $S\subset\cP^+_n$ be a compact set. Define
\[
\cH(S) = \{A\in\cH_n~|~H_A\cap S\neq\emptyset\}.
\]

\begin{proposition}\label{prop:accumulation points}
For any compact set $S$ in $\cP^+_n$, the set $\cH(S)$ is finite.
\end{proposition}
\begin{proof}
For any $\om\in S$, the restriction of the intersection form $Q$ on $\om^\perp$ is negative definite. Let $b_\om(x,y):=-Q(x,y)$ and $q_\om(v)=-Q(v,v)$ denote the induced inner product and quadratic form on $\om^\perp$. Let 
\[
\parallel v\parallel_e^2 = v_0^2+\sum_i v_i^2
\]
be the Euclidean norm. By the light-cone lemma, and by compacity of $S$, there exists $m=m(S)>0$ such that 
\[
q_\om(v)\geq m\parallel v\parallel_e^2\quad\text{for all~} \om\in S\text{~and all~} v\in\om^\perp.
\]
For $\om\in S$, write $K_n=\kappa_\om\om + K^0(\om)$ where
\[
\kappa_\om:=\frac{\om\cdot K_n}{\om^2},\quad K^0(\om):= K_n-\kappa_\om\om.
\]
Then $B=B(S):=\max_{\om\in S} \sqrt{q_\om(K^0(\om))}$ is finite. Now pick $A\in \cH(S)$ and suppose that $\om\in H_A\cap S$. Then,
\[
A\cdot K_n = A\cdot K^0(\om) = -b_\om(A,K^0(\om))
\]
and, by Cauchy-Schwarz,
\[
|A\cdot K_n| \leq \sqrt{q_\om(A)}\sqrt{q_\om(K^0(\om))} \leq \sqrt{q_\om(A)} B.
\]
The adjunction inequality yields
\[
-2 \leq A^2+A\cdot K_n \leq -q_\om(A) + \sqrt{q_\om(A)} B
\]
This last inequality imposes an upper bound $T_\om>0$ on the norm $\sqrt{q_\om(A)}$, so that
\[
m\parallel A\parallel_e^2 ~\leq~ q_\om(A) ~\leq~ T^2.
\]
Because $S$ is compact, we get a uniform upper bound $T=T(S)$ for the Euclidean norm of any $A\in\cH$ whose associated wall $H_A$ intersects $S$. Since $\cH$ is discrete, this completes the proof.
\end{proof}

\begin{corollary}
The hyperplane arrangement $\{H_A\}_{A\in\cH}$ is locally finite in $\cC_n$. Moreover, if $\om_{\B\delta}\in \del\cC_n$ is an accumulation point of a sequence of hyperplanes, then $\om_{\B\delta}$ belongs to the round boundary $\sum_i\delta_i^2 =1$ on which the symplectic volume vanishes.\qed
\end{corollary}

\section{Applications to symplectomorphism groups}\label{S:applications}
In this section we discuss applications to the topology of symplectomorphism groups of blow-ups of $\cp^2$.

\begin{lemma}\label{lemma:indecomposability}
Let $\widetilde{M}_{\B\delta}$ be the symplectic blow-up of $\cp^2$ at $n\geq 3$ balls of equal capacities $\delta_i=\delta<1/3$. Then, for any tamed $J$, the exceptional classes $E_1,\ldots,E_n$ are $J$-indecomposable. 
\end{lemma}
\begin{proof}
Under the convention $\om_{FS}(L)=1$, the condition on the capacities implies that the blow-up form $\om_{\B\delta}$ is reduced. From~\cite[Lemma 2.1]{KKP-2015}, it follows that the classes $E_i$ minimize the symplectic area among all symplectic exceptional classes on $\widetilde{M}_{\B\delta}$. The indecomposability then follows from~\cite[Lemma 1.2]{Pin-MaxTori}.
\end{proof}

\begin{corollary}\label{cor:equivalence Stab=Symp}
Consider an embedding $\phi:B^4(\delta)\sqcup\cdots\sqcup B^4(\delta)\into \cp^2$ of $n\geq 3$ balls of equal capacities $\delta<1/3$. Then the stabilizer subgroup $\Stab(\phi(B^4(\B\delta)))$ is homotopy equivalent to the group $\Symp_h(\widetilde{M}_{\B\delta}, \om_{\B\delta})$ of symplectomorphisms of the blow-up $\widetilde{M}_{\B\delta}$ acting trivially on homology. 
\end{corollary}
\begin{proof}
By~\cite[Section 2]{LP2004}, the stabilizer subgroup $\Stab(\phi(B^4(\B\delta)))$ is homotopy equivalent to the group $\Symp_h(\B\delta,\Sigma)$ preserving each exceptional divisor $\Sigma_i$. This later group is the stabilizer of the transitive action of $\Symp_h(\widetilde{M}_{\B\delta}, \om_{\B\delta})$ on the space $\cC([\Sigma])$ of symplectic configurations of disjoint exceptional divisors in class $[\Sigma]=[\Sigma_1] \sqcup \cdots \sqcup [\Sigma_n]$. This space is homotopy equivalent to the space $\cJ([\Sigma])$ of tame almost complex structures for which $[\Sigma]$ is represented by embedded $J$-holomorphic spheres. By Lemma~\ref{lemma:indecomposability}, this space is the whole space $\cJ(\om_{\B\delta})$ of tame structures, which is contractible. We thus obtain a fibration with contractible base
\begin{equation}
\Symp(\B\delta, \Sigma)\to \Symp_h(\widetilde{M}_{\B\delta}, \om_{\B\delta})\to \cC([\Sigma]).\label{eq:fibration B}
\end{equation}
from which the statement follows.
\end{proof}

\begin{proposition}\label{prop:homotopy type Symp}
Let $\widetilde{M}_{\B\delta}$ be the symplectic blow-up of $\cp^2$ at $n\geq 3$ balls of equal capacities $\delta_i=\delta<1/n$. Then the symplectomorphism group $\Symp_h(\widetilde{M}_{\B\delta}, \om_{\B\delta})$ is homotopy equivalent to the stabilizer of $n$ distinct points in $\cp^2$. Equivalently, $\Symp_h(\widetilde{M}_{\B\delta}, \om_{\B\delta})$ is homotopy equivalent to the homotopy fibre of the action map $\PU(3)\to\Conf(n,\cp^2)$. In particular, $\Symp_h(\widetilde{M}_{\B\delta}, \om_{\B\delta})$ is connected.
\end{proposition}
\begin{proof}
By Theorem~\ref{T:main}, the condition $\delta_i=\delta<1/n$ implies that the embedding space $\IEmb(\B\delta,\cp^2)$ is homotopy equivalent to the ordered configuration space $\Conf(n,\cp^2)$. The homotopy fibration
\begin{equation}
\Symp(\B\delta, \Sigma)\simeq \Stab(\phi(B^4(\B\delta))) \to \Symp(\cp^2)\to \IEmb(\B\delta,\cp^2)\label{eq:fibration A}\\
\end{equation}
implies the stabilizer of $n$ balls is homotopy equivalent to the stabilizer of a configuration of $n$ distinct points. The homotopy equivalence follows from Corollary~\ref{cor:equivalence Stab=Symp}. Connectedness follows from the facts that $\Symp(\cp^2)\simeq \PU(3)$ is connected and that $\Conf(n,\cp^2)$ is simply connected.
\end{proof}

\begin{proposition}\label{prop:IEmb simply connected}
Let $3\leq n\leq 9$ and let  $\B\delta$ be a set of admissible capacities in the reduced chamber. Then the space $\IEmb(\B\delta,\cp^2)$ of symplectic embeddings is simply connected.
\end{proposition}
\begin{proof}
For $\sum_i\delta_i<1/n$, $\IEmb(\B\delta,\cp^2)\simeq \Conf(n,\cp^2)$ which is simply connected. As the reduced chamber does not contain any $(-2)$-wall, the only walls crossed as $\B\delta$ varies within de reduced chamber correspond to strata of codimension at least $4$. Therefore, since stability holds for $3\leq n\leq 9$, the homotopy groups $\pi_k \IEmb(\B\delta,\cp^2)$, $k=1,2$, are constant. The result follows.
\end{proof}

\begin{corollary}\label{cor:connectedness Symp}
For $3\leq n\leq 9$ and $\B\delta$ in the (open) reduced chamber, the symplectomorphism group of the blow-up $\Symp_h(\widetilde{M}_{\B\delta}, \om_{\B\delta})$ is connected.
\end{corollary}
\begin{proof}
Under the hypotheses, the long exact homotopy sequence associated to fibration~\eqref{eq:fibration A} implies that $\Symp(\B\delta, \Sigma)$ is connected. Fibration~\eqref{eq:fibration B} gives
\[
\cdots \to \pi_0 \Symp(\B\delta, \Sigma) \to \pi_0\Symp_h(\widetilde{M}_{\B\delta}, \om_{\B\delta})\to \pi_0\cC([\Sigma])=\{1\}
\]
and the result follows.
\end{proof}

\begin{remark}
Corollary~\ref{cor:connectedness Symp} confirms a recent (and more general) result of Li, Li, and Wu~\cite{LLW-Torelli}. In their terminology, the condition \say{$\B\delta$ is in the open reduced chamber} implies that the blow-up form $\om_{\B\delta}$ is of type $\mathbb{A}$ and that the group $\Symp_h(\widetilde{M}_{\B\delta}, \om_{\B\delta})$ is connected. See Definition~\ref{def:types} and Theorem~\ref{thm:LLW} below.\eoe
\end{remark}

\section{Non simply connected embedding spaces}\label{S:example}

Although we are not able to prove that the homotopy type of $\IEmb(\B\delta,\cp^2)$ changes each time the capacities $\B\delta$ cross a wall in $\cC_9$, we can show that it does change after crossing a finite number of $(-2)$-walls. To do this, we combine our results with the recent work~\cite{LLW-Torelli} of Li, Li, and Wu on the symplectic Torelli group of symplectic rational surfaces.

\begin{definition}[{\cite[Definition 2.13]{LLW-Torelli}}]\label{def:types}
Let $n\geq 5$ and let $\B\delta = \left(\delta_1,\ldots,\delta_n\right)$ be an admissible set of capacities in the reduced chamber of $\cC_n$. 
\begin{itemize}
\item The set $\B\delta$ is of type $\mathbb{E}_6,\mathbb{E}_7,\mathbb{E}_8$ if  
\[
\B\delta=\left(\underbrace{ \frac{1}{3}, \frac{1}{3}, \cdots \frac{1}{3}}_k, m_{k+1} \cdots\right),
\]
where $m_{k+1}<\frac{1}{3}$, $k =6,7,8,$ respectively.
 
\item The set $\B\delta$ is of type $\mathbb{D}_k$ if
\[
\B\delta=\left(a,  \underbrace{\frac{1-a}{2}, \frac{1-a}{2}, \frac{1-a}{2}, \cdots \frac{1-a}{2}}_k, m_{k+2} \cdots\right),
\]
where $ \frac{1-a}{2}> m_{k+2}$, and either  $\frac{1}{3} < a <1 $  and $ k\geqslant 4$; or $a=\frac{1}{3}$ and  $k=4$.
 
\item  The set $\B\delta$ is of type $\mathbb{A}$ for all other possibilities. 
\end{itemize}
\end{definition}

\begin{theorem}[{\cite[Theorem 1.5]{LLW-Torelli}}]\label{thm:LLW}
Let $n\geq 5$ and let $\B\delta$ be an admissible set of capacities in the reduced chamber. Let $\om_{\B\delta}$ be the blow-up form and suppose $\om_{\B\delta}\cdot K_n<0$.
\begin{itemize}
\item If $\B\delta$ is of type $\mathbb{A}$, then $\Symp_h(\widetilde{M}_{\B\delta},\om_{\B\delta})$ is connected.
\item If $\B\delta$ is of type $\mathbb{D}_k$, then $\Symp_h(\widetilde{M}_{\B\delta},\om_{\B\delta})$ is isomorphic to the pure braid group $\PB_k(S^2)$.
\end{itemize}
\end{theorem}

\begin{example}\label{ex:non simply connected 1}
Let $\B\delta$ be a set of type $\mathbb{D}_k$ in $\cC_n$, $n\geq 5$, and suppose $\om_{\B\delta}\cdot K_n<0$. The fibration~\eqref{eq:fibration B} yields
\[
\pi_0 \Symp(\B\delta, \Sigma) \to \pi_0\Symp_h(\widetilde{M}_{\B\delta}, \om_{\B\delta})\simeq \PB_k(S^2) \to \pi_0\cC([\Sigma])=\{1\}
\]
showing that $\pi_0 \Symp(\B\delta, \Sigma)$ surjects onto $\PB_k(S^2)$. Then, the fibration~\eqref{eq:fibration A} gives the exact sequence
\[
\cdots\to \pi_1\PU(3)\simeq \Z_3 \to \pi_1 \IEmb(\B\delta,\cp^2) \to \pi_0 \Symp(\B\delta, \Sigma) \to \{1\}
\]
which implies that there exists a surjection $\pi_1 \IEmb(\B\delta,\cp^2)\to \PB_k(S^2)$.\eoe
\end{example}

In the case $5\leq n\leq 9$, stability within chambers allows for a refinement of the previous example.

\begin{example}\label{ex:non simply connected 2}
For $\epsilon>0$ small enough, consider the set of type $\mathbb{D}_4$ in $\cC_9$ given by
\[
\B\delta_\epsilon := \big(1/3+2\epsilon, 1/3-\epsilon, 1/3-\epsilon, 1/3-\epsilon, 1/3-\epsilon, 1/3-2\epsilon, 1/3-3\epsilon, 1/3-4\epsilon, 1/3-5\epsilon \big).
\]
By Example~\ref{ex:non simply connected 1}, we know that the embedding space $\IEmb(\B\delta_\epsilon,\cp^2)$ is not simply connected. 

By construction, $\B\delta_\epsilon$ belongs to a facet of some stability chamber. Note that for $\epsilon$ sufficiently small, $\B\delta_\epsilon$ is close to the monotone point. in particular, we can assume the areas of all the negative classes with $c_1(A)<0$ listed in Lemma~\ref{lemma:NegativeClasses_9balls} are strictly negative. Moreover, looking at the $(-2)$-classes of the list (with $c_1(A)=0$), it is easy to see that we can choose $\epsilon$ so small that the only $(-2)$-classes that have zero area with respect to $\B\delta_\epsilon$ are of the form $L-E_1-E_j-E_k$, $j\neq k$, $j,k\in\{2,3,4,5\}$. We can then perturb $\B\delta_\epsilon$ by choosing some $\eta>0$ so that the set
\[
\B\delta_{\epsilon,\eta} := \big(1/3+2\epsilon+\eta, 1/3-\epsilon, 1/3-\epsilon, 1/3-\epsilon, 1/3-\epsilon, 1/3-2\epsilon, 1/3-3\epsilon, 1/3-4\epsilon, 1/3-5\epsilon \big)
\]
belongs to the interior of the chamber $\Delta_{\epsilon,\eta}$ characterized by the negativity of the six classes $L-E_1-E_j-E_k$ described above, and by the positivity of all other $(-2)$-classes.

We now describe a walk starting at the reduced chamber, where  $\IEmb(\B\delta,\cp^2)$ is simply connected, to the chamber $\Delta_{\epsilon,\eta}$ where the fundamental group of $\IEmb(\B\delta,\cp^2)$ is non-trivial. We use the notation $L_{ijk}:= L - E_i-E_j-E_k$. 

We start in the reduced chamber, represented in the diagram below. Recall that $r_0= L_{123}$ and $r_i= E_i -E_{i-1}$, with $1 \leq i \leq 8$ and all $(-2)$-curves are positive  in the reduced chamber. 
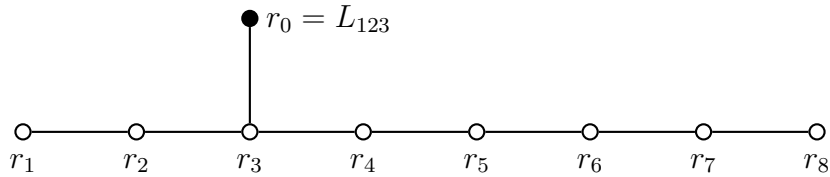
\begin{center}
\begin{tikzpicture}[
    every node/.style={draw, circle, thick, inner sep=0pt, minimum size=2mm, fill=white},
    thick
]
    \def\dx{1.5}   
    \def\dy{1.5}   

    \foreach \i in {1,...,8}
        \node (n\i) at (\i*\dx, 0) {};

    \node (n0) at (3*\dx, \dy) {};
    
    \foreach \i in {1,...,7} {
        \pgfmathtruncatemacro{\j}{\i+1}
        \draw (n\i) -- (n\j);
    }

    \draw (n3) -- (n0);

    \foreach \i in {1,...,8}
        \node[draw=none, below=2pt] at (n\i.south) {$r_\i$};
    \node[draw=none, right=2pt] at (n0.east) {$r_0 = L_{123}$};
    \fill (n0) circle (1mm);

\end{tikzpicture}
\end{center}
Then we perform a reflection across the wall corresponding to $L_{123}$ (which we represent by a solid dot), so this curve becomes negative in the next chamber and the node $r_3$ is replaced by $r_0+r_3= L_{124}$:

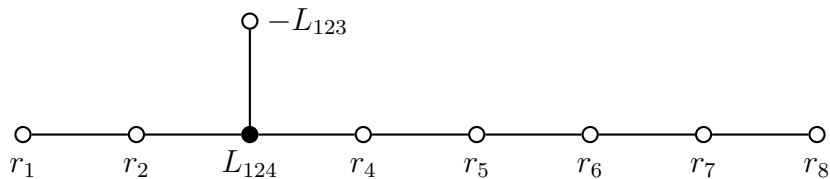
\begin{center}
\begin{tikzpicture}[
    every node/.style={draw, circle, thick, inner sep=0pt, minimum size=2mm, fill=white},
    thick
]
    \def\dx{1.5}   
    \def\dy{1.5}   
    \def\ly{-0.5}  

    \node[draw=none, anchor=base] at (1*\dx, \ly) {$r_1$};
    \node[draw=none, anchor=base] at (2*\dx, \ly) {$r_2$};
    \node[draw=none, anchor=base] at (3*\dx, \ly) {$L_{124}$};
    \node[draw=none, anchor=base] at (4*\dx, \ly) {$r_4$};
    \node[draw=none, anchor=base] at (5*\dx, \ly) {$r_5$};
    \foreach \i in {6,...,8}
        \node[draw=none, anchor=base] at (\i*\dx, \ly) {$r_\i$};
    \node[draw=none, right=2pt] at (n0.east) {$-L_{123}$};
    
    \foreach \i in {1,...,8}
        \node (n\i) at (\i*\dx, 0) {};
    \node (n0) at (3*\dx, \dy) {};
    \foreach \i in {1,...,7} {
        \pgfmathtruncatemacro{\j}{\i+1}
        \draw (n\i) -- (n\j);
    }
    \draw (n3) -- (n0);
\fill (n3) circle (1mm);
\end{tikzpicture}
\end{center}

The next step is to reflect across the wall corresponding to $L_{124}$, so the nodes $r_2$, $r_4$ and $-r_0=-L_{123}$ are replaced by $r_0+r_2+r_3 = L_{134}$, $r_0+r_3+r_4=L_{125}$ and $r_3$, respectively: 

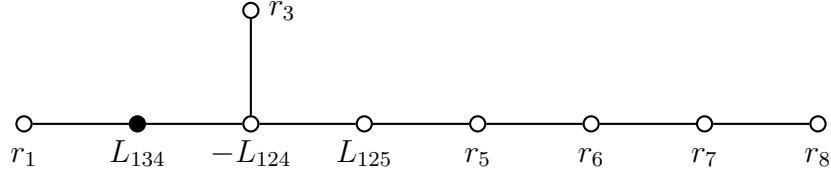
\begin{center}
\begin{tikzpicture}[
    every node/.style={draw, circle, thick, inner sep=0pt, minimum size=2mm, fill=white},
    thick
]
    \def\dx{1.5}   
    \def\dy{1.5}   
    \def\ly{-0.5}  

    \node[draw=none, anchor=base] at (1*\dx, \ly) {$r_1$};
    \node[draw=none, anchor=base] at (2*\dx, \ly) {$L_{134}$};
    \node[draw=none, anchor=base] at (3*\dx, \ly) {$-L_{124}$};
    \node[draw=none, anchor=base] at (4*\dx, \ly) {$L_{125}$};
    \node[draw=none, anchor=base] at (5*\dx, \ly) {$r_5$};
    \foreach \i in {6,...,8}
        \node[draw=none, anchor=base] at (\i*\dx, \ly) {$r_\i$};
    \node[draw=none, right=2pt] at (n0.east) {$r_3$};
    
    \foreach \i in {1,...,8}
        \node (n\i) at (\i*\dx, 0) {};
    \node (n0) at (3*\dx, \dy) {};
    \foreach \i in {1,...,7} {
        \pgfmathtruncatemacro{\j}{\i+1}
        \draw (n\i) -- (n\j);
    }
    \draw (n3) -- (n0);
\fill (n2) circle (1mm);
\end{tikzpicture}
\end{center}

We proceed similarly by reflecting across the wall corresponding to the solid node in each diagram, which yields the following sequence of diagrams.

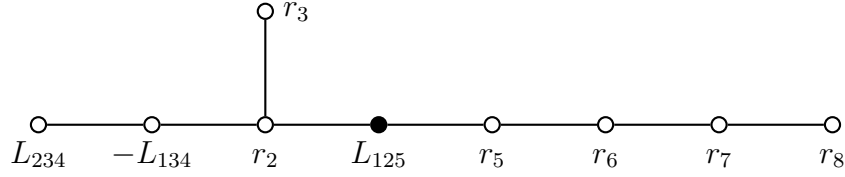
\begin{center}
\begin{tikzpicture}[
    every node/.style={draw, circle, thick, inner sep=0pt, minimum size=2mm, fill=white},
    thick
]
    \def\dx{1.5}   
    \def\dy{1.5}   
    \def\ly{-0.5}  

    \node[draw=none, anchor=base] at (1*\dx, \ly) {$L_{234}$};
    \node[draw=none, anchor=base] at (2*\dx, \ly) {$-L_{134}$};
    \node[draw=none, anchor=base] at (3*\dx, \ly) {$r_2$};
    \node[draw=none, anchor=base] at (4*\dx, \ly) {$L_{125}$};
    \node[draw=none, anchor=base] at (5*\dx, \ly) {$r_5$};
    \foreach \i in {6,...,8}
        \node[draw=none, anchor=base] at (\i*\dx, \ly) {$r_\i$};
    \node[draw=none, right=2pt] at (n0.east) {$r_3$};
    
    \foreach \i in {1,...,8}
        \node (n\i) at (\i*\dx, 0) {};
    \node (n0) at (3*\dx, \dy) {};
    \foreach \i in {1,...,7} {
        \pgfmathtruncatemacro{\j}{\i+1}
        \draw (n\i) -- (n\j);
    }
    \draw (n3) -- (n0);
\fill (n4) circle (1mm);
\end{tikzpicture}
\end{center}

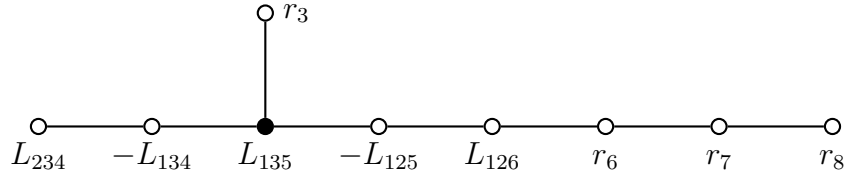
\begin{center}
\begin{tikzpicture}[
    every node/.style={draw, circle, thick, inner sep=0pt, minimum size=2mm, fill=white},
    thick
]
    \def\dx{1.5}   
    \def\dy{1.5}   
    \def\ly{-0.5}  

    \node[draw=none, anchor=base] at (1*\dx, \ly) {$L_{234}$};
    \node[draw=none, anchor=base] at (2*\dx, \ly) {$-L_{134}$};
    \node[draw=none, anchor=base] at (3*\dx, \ly) {$L_{135}$};
    \node[draw=none, anchor=base] at (4*\dx, \ly) {$-L_{125}$};
    \node[draw=none, anchor=base] at (5*\dx, \ly) {$L_{126}$};
    \foreach \i in {6,...,8}
        \node[draw=none, anchor=base] at (\i*\dx, \ly) {$r_\i$};
    \node[draw=none, right=2pt] at (n0.east) {$r_3$};
    
    \foreach \i in {1,...,8}
        \node (n\i) at (\i*\dx, 0) {};
    \node (n0) at (3*\dx, \dy) {};
    \foreach \i in {1,...,7} {
        \pgfmathtruncatemacro{\j}{\i+1}
        \draw (n\i) -- (n\j);
    }
    \draw (n3) -- (n0);
\fill (n3) circle (1mm);
\end{tikzpicture}
\end{center}

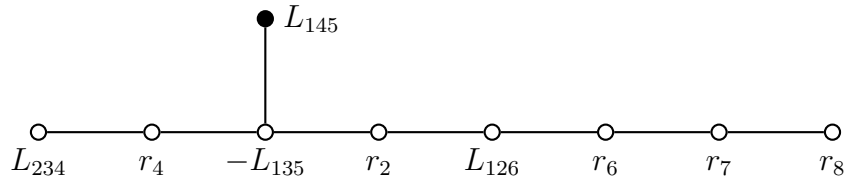
\begin{center}
\begin{tikzpicture}[
    every node/.style={draw, circle, thick, inner sep=0pt, minimum size=2mm, fill=white},
    thick
]
    \def\dx{1.5}   
    \def\dy{1.5}   
    \def\ly{-0.5}  

    \node[draw=none, anchor=base] at (1*\dx, \ly) {$L_{234}$};
    \node[draw=none, anchor=base] at (2*\dx, \ly) {$r_4$};
    \node[draw=none, anchor=base] at (3*\dx, \ly) {$-L_{135}$};
    \node[draw=none, anchor=base] at (4*\dx, \ly) {$r_2$};
    \node[draw=none, anchor=base] at (5*\dx, \ly) {$L_{126}$};
    \foreach \i in {6,...,8}
        \node[draw=none, anchor=base] at (\i*\dx, \ly) {$r_\i$};
    \node[draw=none, right=2pt] at (n0.east) {$L_{145}$};
    
    \foreach \i in {1,...,8}
        \node (n\i) at (\i*\dx, 0) {};
    \node (n0) at (3*\dx, \dy) {};
    \foreach \i in {1,...,7} {
        \pgfmathtruncatemacro{\j}{\i+1}
        \draw (n\i) -- (n\j);
    }
    \draw (n3) -- (n0);
\fill (n0) circle (1mm);
\end{tikzpicture}
\end{center}

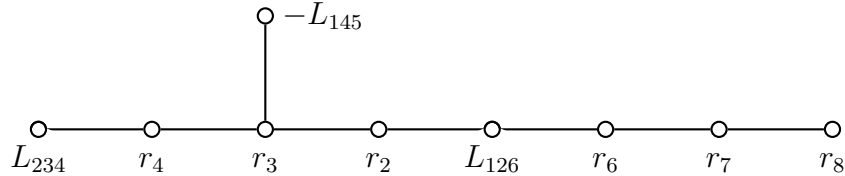
\begin{center}
\begin{tikzpicture}[
    every node/.style={draw, circle, thick, inner sep=0pt, minimum size=2mm, fill=white},
    thick
]
    \def\dx{1.5}   
    \def\dy{1.5}   
    \def\ly{-0.5}  
    \foreach \i in {1,...,8}
        \node (n\i) at (\i*\dx, 0) {};
    \node (n0) at (3*\dx, \dy) {};
    \foreach \i in {1,...,7} {
        \pgfmathtruncatemacro{\j}{\i+1}
        \draw (n\i) -- (n\j);
    }
    \draw (n3) -- (n0);
    \node[draw=none, anchor=base] at (1*\dx, \ly) {$L_{234}$};
    \node[draw=none, anchor=base] at (2*\dx, \ly) {$r_4$};
    \node[draw=none, anchor=base] at (3*\dx, \ly) {$r_3$};
    \node[draw=none, anchor=base] at (4*\dx, \ly) {$r_2$};
    \node[draw=none, anchor=base] at (5*\dx, \ly) {$L_{126}$};
    \foreach \i in {6,...,8}
        \node[draw=none, anchor=base] at (\i*\dx, \ly) {$r_\i$};
    \node[draw=none, right=2pt] at (n0.east) {$-L_{145}$};
    \draw (n1) circle (1mm);
    \draw (n5) circle (1mm);
\end{tikzpicture}
\end{center}

Finally we arrive at the chamber $\Delta_{\epsilon,\eta}$. As explained above $\B{\delta_{\epsilon,\eta}}$ belongs to this chamber and therefore $\IEmb(\B{\delta_{\epsilon,\eta}}, \cp^2)$ is not simply connected.\eoe
\end{example}

\begin{remark}
Note that Example~\ref{ex:non simply connected 2} applies to any $5\leq n\leq 9$ (and to all $n\geq 5$ if we assume stability holds), showing that $\IEmb(\B\delta,\cp^2)$ is not simply connected for suitable choices of capacities. This is in sharp contrast with the cases $1\leq n\leq 4$ for which $\IEmb(\B\delta,\cp^2)$ is simply connected for all choices of admissible capacities, see~\cite{AKP2024}.\eoe
\end{remark}

\appendix
\section{Blow-ups and embeddings}\label{Appendix}
The purpose of this section is to provide the necessary details to complete the proof of Lemma \ref{L:eells}. In order to do that, we need to recall and introduce some notation. In what follows $(M,\omega)$ is a closed rational $4$-manifold. Unless indicated otherwise, we only consider diffeomorphisms acting trivially on homology. This is indicated by the subscript $h$. For example, $\Diff_h$ denotes diffeomorphisms acting trivially on homology; similarly, $\Symp_h$ denotes such a group of symplectomorphisms. Given a topological pair $(X,Y)$ we write $\rel Y$ for objects on $X$ that are standard near $Y$. For instance, $\Diff_h(X,\, \rel Y)$ is the subgroup of diffeomorphisms that are the identity near $Y$, while $\Omega(X,\,\rel Y)$ denotes the space of symplectic forms that are equal to a standard one near $Y$.

To simplify the discussion, we consider embeddings of a single ball. The generalization to $n$ disjoint balls is immediate. Given an admissible capacity $\delta$, the connectedness of $\Emb(\delta,M)$~\cite[Corollary 1.5]{McDuff-DeformationsToIsotopy} implies that the natural action on $\Emb(\delta,M)$ of the group $\Symp_h(M, \omega)$ is transitive, and that evaluation of this action at an embedding $\phi$ defines a fibration
\[
\Stab_h(\phi)\to\Symp_h(M,\om)\to\Emb(\delta,M).
\]
where $\Stab_h(\phi)$ is the group of symplectomorphisms that are the identity on the image $\phi(B^4(\delta))$.

\begin{proposition}[{\cite[Proposition A.11]{C-P-Memoirs}}]\label{prop:equivalence stab and rel B}
For any symplectic embedding 
\[
\phi:B^4(\delta)\into (M,\om),
\]
there is a homotopy equivalence
\[
\Symp_h\Big(M,\, \rel \phi\big(B^4(\delta)\big)\Big) \into \Stab_h(\phi),
\]
where the former denotes the subgroup of symplectomorphisms $\psi$ which are equal to the identity on a neighborhood (varying with $\psi$) of the image $\phi(B^4(\delta))$.
\end{proposition}

Fix a symplectic embedding $\bar{\phi}:B^4(r)\into (M,\om)$, $r>\delta$, extending $\phi$. This extension allows for the definitions of objects that are standard near the image $\phi(B^4(\delta))$. In particular, let $\Omega\Big(\om,\, \rel \phi\big(B^4(\delta)\big)\Big)$ denote the connected component of $\om$ in the space of symplectic forms on $M$ which are cohomologous to $\om$ and which are standard near $\phi(B^4(\delta))$. Similarly, let $\mathcal{P}\Big(\om,\, \rel \phi\big(B^4(\delta)\big)\Big)$ denote the space of tame pairs $(\om,J)$ where $\om\in \Omega\Big(\om,\, \rel \phi\big(B^4(\delta)\big)\Big)$ and where $J$ is standard near the image of $\phi$. Finally, let $\mathcal{A}\Big(\om,\, \rel \phi\big(B^4(\delta)\big)\Big)$ denote the set of almost complex structures belonging to such a pair. The two projections 
\[
\mathcal{A}\Big(\om,\, \rel \phi\big(B^4(\delta)\big)\Big) \xleftarrow{\pi_{\cA}} \mathcal{P}\Big(\om,\, \rel \phi\big(B^4(\delta)\big)\Big) \xrightarrow{\pi_{\Omega}} \Omega\Big(\om,\, \rel \phi\big(B^4(\delta)\big)\Big)
\]
are $\Diff_h\Big(M,\, \rel \phi\big(B^4(\delta)\big)\Big)$-equivariant homotopy equivalences. The stabilizer subgroup $\Symp_h\Big(M,\, \rel \phi\big(B^4(\delta)\big)\Big)$ can be investigated using the homotopy commutative diagram
\begin{equation}\label{eq:basic diagram rel to balls}
\begin{tikzcd}
\Symp_h\Big(M,\, \rel \phi\big(B^4(\delta)\big)\Big)\ar[r] & \Diff_h\Big(M,\, \rel \phi\big(B^4(\delta)\big)\Big)\ar[r]\ar[d,equal] & \Omega\Big(\om,\, \rel \phi\big(B^4(\delta)\big)\Big)\\
\hofib_\mathcal{P}(\delta) \ar[r]\ar[u, swap, "\simeq"] \ar[d, "\simeq"] & \Diff_h\Big(M,\, \rel \phi\big(B^4(\delta)\big)\Big)\ar[r]\ar[d,equal] & \mathcal{P}\Big(\om,\, \rel \phi\big(B^4(\delta)\big)\Big) \ar[u, "\pi_\Omega" right, "\simeq" left]\ar[d, "\pi_\mathcal{A}" right, "\simeq" left]\\
\hofib_\mathcal{A}(\delta) \ar[r] & \Diff_h\Big(M,\, \rel \phi\big(B^4(\delta)\big)\Big)\ar[r] & \mathcal{A}\Big(\om,\, \rel \phi\big(B^4(\delta)\big)\Big)
\end{tikzcd}
\end{equation}
in which the top row is a genuine fibration, and $\hofib_\mathcal{P}$ and  $\hofib_\mathcal{A}$ are the homotopy fibers of the $\Diff_h$ action on the base points. 

Recall that the symplectic blow-up of size $\delta$ along $\phi$ is defined as
\begin{equation}\label{eq:symplectic blow-up size delta}
\widetilde{M}_\delta:= \Big(M \setminus \phi\big(\mathring{B}^4(\delta)\big)\Big) \Big/ \sim
\end{equation}
where the boundary $S^3$ is quotiented along the fibers of the Hopf fibration $S^1\to S^3\to S^2$. The symplectic form on $M\setminus \phi\big(B^4(\delta)\big)$ extends smoothly to a symplectic form $\tom_\delta$ on $\widetilde{M}_\delta$ for which the exceptional divisor $\Sigma:=S^3/\sim$ has symplectic area $\delta$, see~\cite[Section 7.1]{McD-S-SymplecticTopology}. 

As any tamed pair $(\om,J)$ which is standard near $\phi(B^4(\delta))$ extends canonically to a pair $(\tom, \tJ)$ on $\widetilde{M}_\delta$ which is standard near $\Sigma$, the blow-up and blow-down constructions define canonical homeomorphisms
\begin{gather*}
\Omega\Big(\om,\, \rel \phi\big(B^4(\delta)\big)\Big) \cong \Omega\big(\tom_\delta,\, \rel\Sigma\big), \qquad \mathcal{P}\Big(\om,\, \rel \phi\big(B^4(\delta)\big)\Big) \cong \mathcal{P}\big(\tom_\delta,\, \rel\Sigma\big),\\
\mathcal{A}\Big(\om,\, \rel \phi\big(B^4(\delta)\big)\Big) \cong \mathcal{A}\big(\tom_\delta,\, \rel\Sigma\big).
\end{gather*}
Similarly, there is a canonical homeomorphism
\[
\Diff_h\Big(M,\, \rel \phi\big(B^4(\delta)\big)\Big) \cong \Diff_h\Big(\widetilde{M}_\delta,\, \rel \Sigma \Big)
\]
so that the symplectic blow-up construction applied to the diagram~\eqref{eq:basic diagram rel to balls} yields the diagram 
\begin{equation}\label{eq:basic diagram rel to divisors}
\begin{tikzcd}
\Symp_h(\widetilde{M}_\delta,\, \rel \Sigma)\ar[r] & \Diff_h(\widetilde{M}_\delta,\, \rel \Sigma)\ar[r]\ar[d,equal] & \Omega(\tom_\delta,\, \rel\Sigma)\\
\hofib_\mathcal{P}(\delta) \ar[r]\ar[u, swap, "\simeq"] \ar[d, "\simeq"] & \Diff_h(\widetilde{M}_\delta,\, \rel \Sigma)\ar[r]\ar[d,equal] & \mathcal{P}(\tom_\delta,\, \rel\Sigma)\ar[u, "\pi_\Omega" right, "\simeq" left]\ar[d, "\pi_\mathcal{A}" right, "\simeq" left]\\
\hofib_\mathcal{A}(\delta) \ar[r] & \Diff_h(\widetilde{M}_\delta,\, \rel \Sigma)\ar[r] & \mathcal{A}(\tom_\delta,\, \rel\Sigma)
\end{tikzcd}
\end{equation}

Given any pair of capacities $0<\delta<\delta'<r$, there are injections
\begin{gather*}
\mathcal{P}\Big(\om,\, \rel \phi\big(B^4(\delta')\big)\Big)\into \mathcal{P}\Big(\om,\, \rel \phi\big(B^4(\delta)\big)\Big)\\
\Diff_h\Big(M,\, \rel \phi\big(B^4(\delta')\big)\Big)\into \Diff_h\Big(M,\, \rel \phi\big(B^4(\delta)\big)\Big)
\end{gather*}
Passing to the blow-ups yields injective maps
\begin{gather*}
\mathcal{P}\big(\tom_{\delta'},\, \rel \Sigma\big)\into \mathcal{P}\big(\tom_\delta,\, \rel \Sigma\big)\\
\Diff_h\big(\widetilde{M}_{\delta'},\, \rel \Sigma\big)\into \Diff_h\big(\widetilde{M}_\delta,\, \rel \Sigma\big)
\end{gather*}
defined by extending symplectic forms and almost complex structures by standard ones, and by extending diffeomorphisms by the identity. This gives a commutative diagram
\[
\begin{tikzcd}[cramped, row sep=tiny, column sep=tiny]
\Diff_h(\widetilde{M}_{\delta'}, \rel \Sigma) & & \Diff_h(\widetilde{M}_{\delta}, \rel \Sigma) &\\
& \Omega(\tom_{\delta'},\rel \Sigma) & & \Omega(\tom_\delta,\rel \Sigma)\\
\Diff_h(\widetilde{M}_{\delta'}, \rel \Sigma) & & \Diff_h(\widetilde{M}_{\delta}, \rel \Sigma) &\\
& \mathcal{P}(\tom_{\delta'},\rel \Sigma) & & \mathcal{P}(\tom_\delta,\rel \Sigma)\\
\Diff_h(\widetilde{M}_{\delta'}, \rel \Sigma) & & \Diff_h(\widetilde{M}_{\delta}, \rel \Sigma) &\\
& \mathcal{A}(\tom_{\delta'},\rel \Sigma) & & \mathcal{A}(\tom_\delta,\rel \Sigma)
\arrow[from=1-1,to=1-3, hookrightarrow]
\arrow[from=3-1,to=3-3, hookrightarrow]
\arrow[from=5-1,to=5-3, hookrightarrow]
\arrow[from=1-1,to=2-2] 
\arrow[from=3-1,to=4-2]
\arrow[from=5-1,to=6-2]
\arrow[from=1-3,to=2-4]
\arrow[from=3-3,to=4-4] 
\arrow[from=5-3,to=6-4]
\arrow[from=1-1,to=3-1, equal]
\arrow[from=3-1,to=5-1, equal]
\arrow[from=1-3,to=3-3, equal]
\arrow[from=3-3,to=5-3, equal]
\arrow[from=4-2,to=2-2, crossing over]
\arrow[from=4-2,to=6-2, crossing over]
\arrow[from=4-4,to=2-4, crossing over]
\arrow[from=4-4,to=6-4, crossing over]
\arrow[from=2-2,to=2-4, crossing over, hookrightarrow]
\arrow[from=4-2,to=4-4, crossing over, hookrightarrow]
\arrow[from=6-2,to=6-4, crossing over, hookrightarrow]
\end{tikzcd}
\]
where the arrows from left to right are injective. Taking homotopy fibers of the action maps gives a commutative diagram
\[
\begin{tikzcd}
\Symp_h(\widetilde{M}_{\delta'},\rel \Sigma)\ar[r, hookrightarrow] & \Symp_h(\widetilde{M}_\delta,\rel \Sigma) \\
\hofib_{\mathcal{P}}(\delta', \rel \Sigma)\ar[r]\ar[u]\ar[d] & \hofib_{\mathcal{P}}(\delta, \rel \Sigma)\ar[d]\ar[u] \\
\hofib_{\cA}(\delta', \Sigma)\ar[r] & \hofib_{\cA}(\delta, \rel \Sigma)
\end{tikzcd}
\]
in which vertical arrows are homotopy equivalences and where the top horizontal arrow is an inclusion. This makes the following diagram of homotopy fibrations
\begin{equation}\label{homotopy commutative diagram}
\begin{tikzcd}
\Symp_h(\widetilde{M}_{\delta'},\rel \Sigma) \ar[r]\ar[d] & \Diff_h(\widetilde{M}_{\delta'}, \rel \Sigma)\ar[r]\ar[d,hookrightarrow,"\simeq"] &  \mathcal{A}(\tom_{\delta'},\rel \Sigma)\ar[d,hookrightarrow] \\
\Symp_h(\widetilde{M}_{\delta},\rel \Sigma) \ar[r] & \Diff_h(\widetilde{M}_{\delta}, \rel \Sigma)\ar[r] &  \mathcal{A}(\tom_\delta,\rel \Sigma)
\end{tikzcd}
\end{equation}
homotopy commutative. However, the two homotopy fibrations are defined on different symplectic manifolds. To fix this, we construct a diffeomorphism between $\widetilde{M}_{\delta'}$ and $\widetilde{M}_{\delta}$ whose homotopy class is canonical. Let
\[
\rho:\left[\sqrt{\delta'/\pi},\sqrt{r/\pi}\right]\to\left[\sqrt{\delta/\delta'},1\right]
\]
be a non-decreasing function which is equal to $\lambda:=\sqrt{\delta/\delta'}$ near $\sqrt{\delta'/\pi}$ and to $1$ near $\sqrt{r/\pi}$. The radial diffeomorphism
\begin{align*}
h:B^4(r)\setminus B^4(\delta') &\to B^4(r)\setminus B^4(\delta)\\
z&\mapsto \rho(|z|)z
\end{align*}
extends to a diffeomorphism $\Psi:\widetilde{M}_{\delta'}\to\widetilde{M}_\delta$ which identifies the two exceptional divisors, which is the identity outside $\bar{\phi}(B^4(r))$, and which induces a homeomorphism between $\Diff_h(\widetilde{M}_{\delta'}, \rel \Sigma)$ and $\Diff_h(\widetilde{M}_{\delta}, \rel \Sigma)$. Since $\Psi$ is a radial rescaling on $\bar{\phi}(B^4(r))$, it intertwines the standard complex structures near the exceptional divisors and it pulls back $\tom_\delta$ to a form on $\widetilde{M}_{\delta'}$ that still tames $J_0$. Moreover, since $\Psi$ is just multiplication by the constant $\lambda$ near $\phi(B^4(\delta'))$,  $\Psi^*\tom_\delta$ is equal to $\lambda^2\tom_{\delta'}$ near $\Sigma$. It follows that $\Psi$ takes tame pairs in $\mathcal{P}(\delta,\,\rel \Sigma)$ to the space $\mathcal{P}(\Psi^*\tom_\delta,\,\rel \Sigma)$ of tame pairs on $\widetilde{M}_{\delta'}$ in which the symplectic form is equal to $\lambda^2\tom_{\delta'}$ near $\Sigma$. 

\begin{lemma}
Given two admissible capacities $\delta<\delta'<r$, the tame-to-tame inflation along $\Sigma$ defines an inclusion
\begin{equation}
\mathcal{A}(\tom_{\delta'},\, \rel \Sigma)\into \mathcal{A}\left(\Psi^*\tom_\delta,\, \rel \Sigma\right).
\label{prop:inclusion of spaces of ACS}
\end{equation}
\end{lemma}
\begin{proof}
The tame-to-tame inflation along $\Sigma$ consists in adding to a symplectic form $\om$ an appropriate Thom form $\tau$ representing the Poincaré dual of $\Sigma$, see~\cite[Section 3.2]{CPP-JTtameInflation}. Since $\Sigma$ is an exceptional sphere, and since the initial pair $(\om, J)$ is standard near $\Sigma$, we can work in a standard neighborhood $L(\delta)$ of the zero section in $\mathcal{O}(-1)$. Let $s$ be the radial distance, let $d\theta$ be the angular form, and let $f$ be a non-decreasing function which is equal to $-1$ near $0$ and which vanishes for $s>\delta-\epsilon$. Define
\[
\tau = f'(s) ds\wedge d\theta + f(s)\om_{FS}
\]
where $\om_{FS}$ is the Fubini-Study form on $\cp^1$ of total area one.
By construction, $\tau=-\om_{FS}$ near $\Sigma$, so that adding a multiple of $\tau$ to the model forms 
\[
\om_\alpha :=\pi^*\om_0 + \pr^*\alpha\om_{FS}
\]
amounts, locally near $\Sigma$, to decreasing $\alpha$. In particular, taking $\alpha=\delta'$, we can inflate to get the standard form $\om_\delta$ near $\Sigma$. A direct computation shows that the inflated forms all tame $J$, see~\cite[Section 3.2]{CPP-JTtameInflation}. 
\end{proof}

Identify $\widetilde{M}_{\delta'}$ with the complex blow-up of $M$ at the center $\phi(0)$, $\widetilde{M}$. Then, using the notation of Section~\ref{S:preliminaries}, we have identifications
\begin{gather}
\mathcal{A}(\delta', \rel \Sigma)=\mathcal{A}(\tom_{\delta'}, \rel \Sigma),\quad 
\mathcal{A}(\delta, \rel \Sigma)=\mathcal{A}\left(\Psi^*\tom_\delta,\, \rel \Sigma\right),\label{eq:identification A}\\
\Symp_h(\delta,\rel \Sigma)=\Symp_h(\Psi^*\tom_\delta,\rel \Sigma).\notag
\end{gather}
Combining the preceding observations, we obtain the following proposition.
\begin{proposition}\label{prop:induced map between Symp groups}
Given a symplectic embedding $\bar{\phi}: B^4(r)\into (M,\om)$ and two capacities $\delta<\delta'<r$, the symplectic blow-up construction together with the diffeomorphism $\Psi:\widetilde{M}_{\delta'}\to\widetilde{M}_\delta$ induces an inclusion
\[
\mathcal{A}(\delta', \rel \Sigma) \into \mathcal{A}(\delta, \rel \Sigma)
\]
and a map
\[
\Symp_h(\delta', \rel \Sigma) \into \Symp_h(\delta, \rel \Sigma)
\]
which makes the diagram \eqref{homotopy commutative diagram} homotopy commutative. Moreover, composing with the homotopy equivalence 
\[
\Symp_h(\delta,\rel \Sigma)\to \Stab_h(\phi(\delta)),
\]
and similarly for $\delta'$, we obtain a homotopy commutative diagram
\[
\begin{tikzcd}
\Symp_h(\delta',\rel \Sigma) \ar[r,hookrightarrow,"\simeq"]\ar[d] & \Stab_h(\phi(\delta'))\ar[r]\ar[d] &  \Symp_h(M,\om)\ar[d,equal]\ar[r] & \Emb(\delta',M)\ar[d]\\
\Symp_h(\delta,\rel \Sigma) \ar[r,"\simeq"] & \Stab_h(\phi(\delta))\ar[r] &  \Symp_h(M,\om)\ar[r] & \Emb(\delta,M).
\end{tikzcd}
\]
\qed
\end{proposition}

\bibliography{Bibliography}
\bibliographystyle{plain}

\end{document}